\documentclass[10pt]{amsart}

\usepackage{amssymb,amsmath,amsthm,amsfonts}
\usepackage{enumerate,comment,hyperref}

\usepackage{tikz}
\usetikzlibrary{positioning}
\usepackage{tikz-cd}
\usetikzlibrary{cd}

\newcommand{\alg}{\mathbf}
\newcommand{\class}{\mathsf}

\newcommand{\set}[2]{\{ #1 \mid #2 \}}
\newcommand{\pair}[2]{\langle #1, #2 \rangle}
\newcommand{\tuple}[1]{\overline{#1}}
\newcommand{\card}[1]{| #1 |}

\newcommand{\assign}{\mathrel{:=}}
\newcommand{\equals}{\approx}
\newcommand{\iso}{\cong}
\newcommand{\seq}{\triangleright}

\newcommand{\into}{\hookrightarrow}
\newcommand{\onto}{\twoheadrightarrow}

\newcommand{\down}{{\downarrow}}
\newcommand{\up}{{\uparrow}}

\newcommand{\Power}{\mathcal{P}}
\newcommand{\UpFin}[1]{\mathrm{Up}_{\omega}(#1)}
\newcommand{\Multi}[1]{\mathbb{N}[#1]}
\newcommand{\Down}[1]{\mathcal{D}(#1)}
\newcommand{\DM}[1]{\Down{\Multi{#1}}}

\newcommand{\0}{\mathsf{0}}
\newcommand{\1}{\mathsf{1}}
\newcommand{\id}{\mathrm{id}}
\newcommand{\sast}{/_{\ast}}

\newcommand{\unit}{\eta}
\newcommand{\embed}{\iota}
\newcommand{\embedof}[1]{\iota_{\scriptscriptstyle#1}}

\DeclareMathOperator{\Mon}{Mon}
\DeclareMathOperator{\End}{End}
\DeclareMathOperator{\Gen}{Gen}
\DeclareMathOperator{\ExpEnd}{ExpEnd}
\DeclareMathOperator{\Con}{Con}
\DeclareMathOperator{\Sg}{Sg}

\newcommand{\Free}[1]{\alg{F}(#1)}
\newcommand{\FreeIn}[2]{\alg{F}_{#1}(#2)}

\newcommand{\Fm}{\alg{Fm}}
\newcommand{\Nat}{\mathbb{N}}
\newcommand{\Two}{\mathbf{2}}

\newcommand{\A}{\alg{A}}
\newcommand{\Ad}{\A_{\mathrm{d}}}
\newcommand{\B}{\alg{B}}
\newcommand{\Bd}{\B_{\mathrm{d}}}
\renewcommand{\P}{\alg{P}}
\newcommand{\Q}{\alg{Q}}
\newcommand{\R}{\alg{R}}
\newcommand{\hd}{h_{\mathrm{d}}}

\newtheorem{theorem}{Theorem}[section]
\newtheorem{lemma}[theorem]{Lemma}
\newtheorem{proposition}[theorem]{Proposition}
\newtheorem{corollary}[theorem]{Corollary}
\newtheorem{fact}[theorem]{Fact}
\newtheorem*{theorem*}{Theorem}
\newtheorem*{corollary*}{Corollary}

\theoremstyle{definition}
\newtheorem{definition}[theorem]{Definition}
\newtheorem{example}[theorem]{Example}

\theoremstyle{remark}

\title{Equivalence of multiset-based consequence~relations}

\author[A. Madanshekaf]{Ali Madanshekaf}
\address{Semnan University, Department of Mathematics}
\email{amadanshekaf@semnan.ac.ir}
\author[A. P\v{r}enosil]{Adam P\v{r}enosil}
\address{Universitat de Barcelona, Departament de Filosofia}
\email{adam.prenosil@gmail.com}
\author[Z. K. Seresti]{Zeinab Khanjanzadeh Seresti}
\address{Semnan University, Department of Mathematics}
\email{z.khanjanzadeh@gmail.com}
\author[C. Tsinakis]{Constantine Tsinakis}
\address{Vanderbilt University, Department of Mathematics}
\email{constantine.tsinakis@vanderbilt.edu}
\thanks{Parts of this research have done while the first-named author was spending his sabbatical leave at the Department of Mathematics,  Vanderbilt University (VU), Nashvile, TN, USA. This author expresses his thanks for the warm hospitality and facilities provided by Prof. Constantine Tsinakis and Department of Mathematics of VU. He is greatly indebted to  Semnan University for its financial support during the sabatical. The work of the second author was funded by the grant 2021 BP 00212 of the grant agency AGAUR of the Generalitat de Catalunya.}

\begin{document}

\begin{abstract}
  The pioneering work of Blok and J\'{o}nsson and its further development by Galatos and Tsinakis initiated an abstract study of consequence relations using the tools of module theory, where consequence relations over all types of syntactic objects are put on an equal footing. However, the assumption that in a consequence relation the premises form merely a set, as opposed to a more complicated structure, is still retained. An attempt to extend this framework to account for inferentially substructural generalizations of consequence relations, where the premises have the structure of a finite multiset, was recently made by Cintula, Gil-F\'{e}rez, Moraschini, and Paoli. In this paper, we develop a different inferentially substructural generalization of the work of Galatos and Tsinakis, where we instead assume that the premises have the structure of a set of finite multisets. This leads a somewhat smoother framework which, unlike that of Cintula et al., covers the original theory of Galatos and Tsinakis as a special case.
\end{abstract}


\maketitle

\section{Introduction}

  The goal of the present paper is to develop an abstract theory of consequence relations in the spirit of the module-theoretic approach of Galatos and Tsinakis~\cite{GT} which would allow for substructural behavior at the inferential level. That is, we aim for an abstract theory of consequence relations where the premises have more structure than merely that of a set: for instance, they may have the structure of a list or a multiset. One such framework was already introduced by Cintula et al.~\cite{CGMP}, who directly replace sets of premises by finite multisets of premises. In this paper, we instead replace sets of premises by \emph{sets of finite multisets} of premises. This leads to a framework which is in some respects smoother than that of Cintula et al.~\cite{CGMP}, while at the same time subsuming it as a special case.

  In this introductory section, we acquaint the reader with the existing abstract approach to set-based consequence relations due to Blok and J\'{o}nsson~\cite{BJ} and Galatos and Tsinakis~\cite{GT}. We then sketch the existing attempt of Cintula et al.~\cite{CGMP} to extend this approach to multiset-based consequence relations, and clarify its relation to the approach taken in the present paper. The rest of the paper is then devoted to rebuilding the theory of Galatos and Tsinakis in the context of multiset-based consequence relations. In contrast to Cintula et al., we recover the original framework of Galatos and Tsinakis as a special (idempotent) case.

\subsection{Consequence relations in abstract algebraic logic}

  The standard point of departure in abstract algebraic logic (see e.g.~\cite{Font16}) is that the basic syntactic objects of a propositional logic have the structure of an algebra, namely they form the algebra of formulas~$\Fm$, which is the nothing but the absolutely free algebra over a given infinite set of generators in a given algebraic signature. A logic is then defined as a \emph{structural consequence relation} on $\Fm$. This is a relation $\vdash$ between sets of formulas and formulas (written as $\Gamma \vdash \varphi$) which satisfies the following conditions:
\begin{enumerate}[(i)]
\item Reflexivity: if $\gamma \in \Gamma \subseteq \Fm$, then $\Gamma \vdash \gamma$.
\item Transitivity: if $\Gamma \vdash \delta$ for each $\delta \in \Delta \subseteq \Fm$ and $\Delta \vdash \varphi$, then $\Gamma \vdash \varphi$.
\item Structurality: if $\Gamma \vdash \varphi$, then $\sigma[\Gamma] \vdash \sigma(\varphi)$ for each substitution~$\sigma$.
\end{enumerate}
  We can equivalently view $\vdash$ as a binary relation on the powerset $\Power(Fm)$ if we define
\begin{align*}
  \Gamma \vdash \Delta \iff \Gamma \vdash \delta \text{ for all } \delta \in \Delta.
\end{align*}
  The relations $\vdash$ of this form on $\Power(Fm)$ are axiomatized by the following conditions:
\begin{enumerate}[(i)]
\item Reflexivity: if $\Delta \subseteq \Gamma \subseteq \Fm$, then $\Gamma \vdash \Delta$,
\item Transivity: if $\Gamma \vdash \Delta$ and $\Delta \vdash \Phi$, then $\Gamma \vdash \Phi$,
\item Collection: if $\Gamma \vdash \Delta_{i}$ for $\Delta_{i} \subseteq \Fm$ with $i \in I$, then $\Gamma \vdash \bigcup_{i \in I} \Delta_{i}$,
\item Structurality: if $\Gamma \vdash \Delta$, then $\sigma[\Gamma] \vdash \sigma[\Delta]$ for each substitution~$\sigma$.
\end{enumerate}

  Each structural consequence relation can equivalently be described as a \emph{structural closure operator} on the powerset $\Power(\Fm)$: a closure operator $C$ on $\Power(\Fm)$ such that if $T \in \Power(\Fm)$ is a closed set of $C$, then so is $\sigma^{-1}[T]$ for each substitution $\sigma$. The bijective correspondence is straightforward:
\begin{align*}
  & C(\Gamma) \assign \set{\varphi \in \Fm}{\Gamma \vdash \varphi}, & & \Gamma \vdash \varphi \iff \varphi \in C(\Gamma).
\end{align*}

  The advantage of the above framework is that it is very easy to deploy: one does not need to digest pages of theory in order to understand what a logic~is. Its disadvantage is that it is restricted to consequence relations specifically over formulas. It therefore does not subsume consequence relations over equations, sequents, or other types of syntactic objects. In~particular, this framework is not suitable for relating consequence relations over different types of syntactic objects, which is precisely what one wants to do when talking about algebraizable logics: algebraizability simply means being in a precise sense equivalent, \emph{qua} a consequence relation, to an equational consequence relation. Granted, the particular topic of algebraizability can be handled in an \emph{ad hoc} way by allowing for consequence relations over $k$-tuples of formulas, which is the route taken by Blok and Pigozzi~\cite{BP01}, but this is not at all a principled solution to the problem of relating consequence relations over different types of syntactic objects.

\subsection{A module-theoretic account of consequence}

  The first key step towards a theory of consequence relations over arbitrary syntactic objects was taken by Blok and J\'{o}nsson~\cite{BJ}. Their insight was that consequence relations can in fact be set up on a much more meager basis than an algebra of formulas $\Fm$. Namely, it suffices to have the structure of a \emph{monoid act} on the syntactic objects: a general monoid~$\alg{M}$ (to be thought of as the monoid of substitutions) with a left action $\ast$ (to be thought of as the application of a substitution) on an arbitrary set~$X$ (to be thought of as the set of formulas). This has the benefit of subsuming structural consequence relations over all sorts of syntactic objects besides formulas, such as equations, sequents, hypersequents, tree-sequents and so on. Such syntactic objects do not naturally come with the structure of an algebra, but they do come with an obvious action of the monoid of substitutions. For example, if the set of sequents $Seq$ consists of pairs of finite tuples of formulas in $Fm$ and the monoid of substitutions $\alg{M}$ consists of the endomorphisms of $\Fm$, then the natural action $\ast$ of $\alg{M}$ on $Seq$ is the following:
\begin{align*}
  & \sigma \ast (\gamma_{1}, \dots, \gamma_{m} \seq \delta_{1}, \dots, \delta_{n}) \assign \sigma(\gamma_{1}), \dots, \sigma(\gamma_{m}) \seq \sigma(\delta_{1}), \dots, \sigma(\delta_{n}).
\end{align*}
  The action allows us to talk about preimages of subsets of $X$ with respect to a substitution, and therefore to define the notion of a structural closure operator. Equivalently, a closure operator $C$ on a set $X$ is structural if $\sigma \ast C(Y) \subseteq C(\sigma \ast Y)$ for each $Y \subseteq X$, where $\sigma \ast Y \assign \set{\sigma \ast y}{y \in Y}$.

  The second key step was taken by Galatos and Tsinakis~\cite{GT}, extending the above work of Blok and J\'{o}nsson. Instead of building a theory of consequence relations around a monoid~$\alg{M}$ acting on a set~$X$, Galatos and Tsinakis build it around a quantale~$\alg{A}$ acting on a complete join semilattice $\alg{P}$. We remind the reader that a (multiplicative) \emph{quantale} has the structure of both a monoid $\langle A, \cdot, \1 \rangle$ and a complete join semilattice $\langle A, \bigvee \rangle$ such that 
\begin{align*}
  & x \cdot \bigvee_{i \in I} y_{i} = \bigvee_{i \in I} (x \cdot y_{i}), & & \left( \bigvee_{i \in I} x_{i} \right) \cdot y = \bigvee_{i \in I} (x_{i} \cdot y).
\end{align*}
  The action of $\alg{A}$ on $\alg{P}$ is required to preserve arbitrary joins in both coordinates:
\begin{align*}
  & \left( \bigvee_{i \in I} a_{i} \right) \ast x = \bigvee_{i \in I} (a_{i} \ast x), & & a \ast \bigvee_{i \in I} x_{i} = \bigvee_{i \in I} (a \ast x_{i}).
\end{align*}
  A complete join semilattice $\alg{P}$ equipped with such an action of a quantale $\alg{A}$ is called an \emph{$\alg{A}$-module} by Galatos and Tsinakis.

  The complete semilattice structure is already implicit in the work of Blok and J\'{o}nsson: an action of a monoid $\alg{M}$ on a set $X$ extends uniquely to an action of $\alg{M}$ on the powerset $\Power(X)$ which commutes with arbitrary unions in the second coordinate. This further extends uniquely to a module action of $\Power(\alg{M})$ on $\Power(X)$, i.e.\ an action which also commutes with arbitrary unions in the first coordinate, if we observe that each monoid $\alg{M}$ naturally yields a quantale $\Power(\alg{M})$ with the operations
\begin{align*}
  & X \cdot Y \assign \set{x \cdot y}{x \in X \text{ and } y \in Y}, & & \bigvee_{i \in I} X_{i} \assign \bigcup_{i \in I} X_{i},
\end{align*}
  The abstraction step of Galatos and Tsinakis then consists in replacing $\Power(\alg{M})$ and $\Power(X)$ by a general quantale $\alg{A}$ and a general complete join semilattice $\alg{P}$.

  This abstraction serves two purposes. Firstly, it allows us to import ideas from the classical theory of modules over rings (such as the description of cyclic projective modules). Secondly, it gives us another equivalent way of looking at logics, namely as \emph{quotients of modules}. Consider an $\alg{A}$-module~$\alg{P}$. A closure operator $\gamma$ on $\alg{P}$ is called \emph{structural} if $a \ast \gamma(x) \leq \gamma(a \ast x)$. The fixpoints of a structural closure operator $\gamma$ on $\alg{P}$ form a complete join semilattice $\alg{P}_{\gamma}$ which inherits the $\alg{A}$-module structure of $\alg{P}$ if we take the action $a \ast_{\gamma} x \assign \gamma(a \ast x)$. The map $\gamma$ is then a surjective homomorphism of $\alg{A}$-modules: a homomorphism of complete join semilattices which commutes with the $\alg{A}$-module action. Conversely, each surjective homomorphism of $\alg{A}$-modules $h\colon \alg{P} \onto \alg{Q}$ determines a structural closure operator on $\alg{P}$, namely
\begin{align*}
  \gamma(x) \assign \bigvee \set{y \in \alg{P}}{h(y) = h(x)} = \bigvee \set{y \in \alg{P}}{h(y) \leq h(x)}.
\end{align*}
  This yields a bijective correspondence between the structural closure operators on an $\alg{A}$-module $\alg{P}$ and the module quotients of $\alg{P}$ in the sense of universal algebra.

  The above framework can now be employed to smoothly handle equivalences between consequence relations over syntactic objects of different types. Consider a quantale $\alg{A}$ and two $\alg{A}$-modules $\alg{P}$ and $\alg{Q}$ (to be thought of as encoding two different types of syntactic objects). For instance, in the context of algebraization of propositional logics, $\alg{A}$ is the powerset quantale of the monoid of substitutions (the endomorphism monoid of~$\Fm$), $\alg{P}$ is the complete join semilattice of sets of formulas $\Power(Fm)$, and $\alg{Q}$ is the complete join semilattice of sets of equations $\Power(Eq)$. Consequence relations on objects of syntactic types $\alg{P}$ and $\alg{Q}$ correspond to structural closure operators $\gamma$ and $\delta$ on $\alg{P}$ and $\alg{Q}$, respectively. For instance, logical consequence relations $\vdash$ correspond to structural closure operators $\gamma$ on $\Power(Fm)$ (namely, closure under logical consequence in $\vdash)$, while equational consequence relations $\vDash$ correspond to structural closure operators $\delta$ on $\Power(Eq)$ (namely, closure under equational consequence in $\vDash$).

  With this intended interpretation of $\alg{P}$ and $\alg{Q}$ and $\gamma$ and $\delta$ in mind, the results of~\cite{GT} concerning equivalences of consequence relations are the following.

\begin{theorem*}
  The following are equivalent for each $\alg{A}$-module $\alg{P}$:
\begin{enumerate}[(i)]
\item Each embedding of modules $f\colon \alg{P}_{\gamma} \into \alg{Q}_{\delta}$ for each pair of structural closure operators $\gamma$ on $\alg{P}$ and $\delta$ on $\alg{Q}$ is induced by a homomorphism of modules $\tau\colon \alg{P} \to \alg{Q}$ in the sense that $f \circ \gamma = \delta \circ \tau$.
\item $\alg{P}$ is a projective module in the sense that each homomorphism of modules $h\colon \alg{P} \to \alg{Q}$ lifts along each surjective homomorphism of modules $g\colon \alg{R} \onto \alg{Q}$ to a homomorphism of modules $h^{\sharp}\colon \alg{P} \to \alg{R}$ such that $h = g \circ h^{\sharp}$:
\[
\begin{tikzcd}
 & & \alg{R} \arrow[d,"g",->>] \\
 & \alg{P} \arrow[r,"h"] \arrow[ur,"h^{\sharp}",dashed] & \alg{Q}
\end{tikzcd}
\]
\end{enumerate}
\end{theorem*}

\begin{corollary*}
  Let $\alg{P}$ and $\alg{Q}$ be projective $\alg{A}$-modules, and let $\gamma$ and $\delta$ be structural closure operators on $\alg{P}$ and $\alg{Q}$, respectively. Then each isomorphism of modules given by $f\colon \alg{P}_{\gamma} \to \alg{Q}_{\delta}$ and $g\colon \alg{Q}_{\delta} \to \alg{P}_{\gamma}$ is induced by a pair of homomorphisms of modules $\tau\colon \alg{P} \to \alg{Q}$ and $\rho\colon \alg{Q} \to \alg{P}$ in the sense that $f \circ \gamma = \tau \circ \delta$ and $g \circ \delta = \gamma \circ \rho$:
\[
\begin{tikzcd}
  & \alg{P} \arrow[r,"\tau",shift left,dashed] \arrow[d,"\gamma",->>] & \alg{Q} \arrow[l,"\rho",shift left,dashed] \arrow[d,"\delta",->>] \\
  & \alg{P}_{\gamma} \arrow[r,"f",shift left] & \alg{Q}_{\delta} \arrow[l,"g",shift left]
\end{tikzcd}
\]
\end{corollary*}

  In other words, in projective modules abstract equivalences between structural closure operators always have concrete syntactic witnesses, namely the maps $\tau$ and~$\rho$. In the case of $\alg{P} \assign \Power(Fm)$ and $\alg{Q} \assign \Power(Eq)$, we get the following result: each isomorphism between lattices of logical and equational theories which is compatible with substitutions arises from some translations $\tau\colon \Power(Fm) \to \Power(Eq)$ and $\rho\colon \Power(Eq) \to \Power(Fm)$ which satisfy the conditions
\begin{align*}
  \Gamma \vdash \Delta & \iff \tau[\Gamma] \vDash \tau[\Delta], & \Gamma \vdash \rho[\tau[\Gamma]] \vdash \Gamma, \\
  E \vDash F & \iff \rho[E] \vdash \rho[F], & E \vDash \tau[\rho[E]] \vDash E.
\end{align*}
  We therefore recover the theorem of Blok and Pigozzi~\cite[Theorem~3.7(ii)]{BP89} stating that a logic $L$ is algebraizable with a class of algebras $\class{K}$ as its equivalent algebraic semantics if and only if there is an isomorphism between the logical theories of $L$ and the equational theories of $\class{K}$ which commutes with inverse substitutions.\footnote{Blok and Pigozzi~\cite{BP89} require $\class{K}$ to be a quasivariety in their formulation of the theorem, but this restriction is in fact immaterial.} Using our current language of modules, we can say that Blok and Pigozzi showed that algebraizability is a purely module-theoretic property.

  In order to apply the above corollary to concrete cases, we need to be able to identify projective modules. The following theorem tells us how to do so, at least in the cyclic case. For example, it identifies $\Power(Fm)$ and $\Power(Eq)$ as cyclic projective modules, and it identifies the module $\Power(Seq)$ of sets of sequents as a projective module by virtue of being a coproduct of projective modules.

  As in ordinary module theory, we say that an $\alg{A}$-module $\alg{P}$ is \emph{cyclic} if there is some $u \in \alg{P}$ such that each $x \in \alg{P}$ has the form $x = a \ast u$ for some $a \in \alg{A}$.

\begin{theorem*}
  The following are equivalent for each $\alg{A}$-module $\alg{P}$:
\begin{enumerate}[(i)]
\item $\alg{P}$ is a cyclic projective module.
\item $\alg{P}$ is isomorphic to the module $\alg{A} \cdot u$ for some idempotent $u \in \alg{A}$.
\end{enumerate}
\end{theorem*}

  This description of cyclic projective modules over a quantale was extended to arbitrary projective modules by Russo~\cite[Theorem~2.4]{Russo16}.

\subsection{Going substructural, take 1: from sets to multisets}

  Having reviewed the existing concrete and abstract approaches to consequence relations, we now introduce the proper topic of this paper: substructural consequence relations.\footnote{It is an unfortunate fact that the term \emph{structural} has two well-established but completely unrelated senses, both of which are important in this paper. In one sense, structurality is the property of being invariant under substitutions. In another sense, structurality is the admissibility of the so-called structural rules of classical or intuitionistic logic in sequent calculi (namely, the rules of Exchange, Contraction, and Weakening). Being \emph{substructural}, in contrast, means that some of these structural rules are not admissible.}

  The idea behind substructural logics is by now a familiar one (see e.g.~\cite{Paoli02}). The substructural charge against classical and intuitionistic logic is that they fail to be resource-conscious: deriving $y$ from, say, $x$ and $x \rightarrow (x \rightarrow y)$ is more demanding in terms of logical resources than deriving $y$ from $x$ and $x \rightarrow y$, since in the former case one has to use the premise $x$ twice. This distinction is not reflected anywhere in classical and intuitionistic logic. In contrast, substructural logics distinguish between two ways of combining propositions: a multiplicative or serial one ($\varphi \cdot \psi$), which allows us to use one proposition and then the other, and an additive or parallel one ($\varphi \wedge \psi$), which allows us to use once either of the two propositions. One then derives $y$ from $x \cdot x$ and $x \rightarrow (x \rightarrow y)$, but in general neither from $x$ and $x \rightarrow (x \rightarrow y)$ nor from $x \wedge x$ and $x \rightarrow (x \rightarrow y)$. This yields a more fine-grained logical analysis of propositions.

  This kind of resource-sensitivity is reflected at the level of sequent calculi by the failure of some of the structural rules of Exchange, Contraction, and Weakening. These rules ensure that in a sequent $\gamma_{1}, \dots, \gamma_{n} \seq \varphi$ the comma can be interpreted simply as the set-theoretic comma combining the premises into a set. In the absence of Contraction or Weakening, in contrast, the premises need to be treated as a multiset, and in the absence of Exchange as a list. One therefore again gets a more fine-grained logical analysis of the ways in which premises can be combined.

  However, despite all this talk about their resource-sensitivity, substructural logics have largely (and, it has to be said, with much success) been studied within the framework of set-based and therefore resource-insensitive consequence relations. While these logics have substructural sequent calculi, the consequence relation normally associated with such a logic is the so-called \emph{external} consequence relation which does not directly reflect the absence of structural rules: for a set $\Gamma \subseteq Fm$
\begin{align*}
  \Gamma \vdash_{\mathrm{e}} \varphi \iff \set{\emptyset \seq \gamma}{\gamma \in \Gamma} \text{ proves } \emptyset \seq \varphi \text{ in the associated sequent calculus}.
\end{align*}
  Substructurality manifests itself at the level of theorems (for instance, the implication $(x \cdot x) \rightarrow x$ may fail to be a theorem) but not at the level of consequence (the rule $x, x \vdash x$ holds). In contrast, the internal consequence relation associated with a sequent calculus is genuinely substructural: for a list of formulas~$\Gamma$
\begin{align*}
  \Gamma \vdash_{\mathrm{i}} \varphi \iff \text{the associated sequent calculus proves } \Gamma \seq \varphi.
\end{align*}
  A natural question is therefore whether one can develop a framework for studying consequence relations such as $\vdash_{\mathrm{i}}$ which reflect the substructural nature of such logics directly at the inferential level, rather than merely at the level of theorems. (For a more thorough and detail motivation and discussion of multiset-based consequence relations, the reader should consult the papers~\cite{CGMP,CintulaPaoli21}.)

  What would such a framework look like? The structures over which such substructural consequence relations operate need to be more resource-sensitive types of mathematical structures than sets, for instance multisets or lists. The framework which we develop in the present paper allows us to handle both multisets and lists, either with or without the rule of Weakening, but for the sake of simplicity let us restrict the current discussion to the framework of multisets with Weakening.

  Let us first review some basic definitions related to multisets and introduce some notation. A multiset $\Gamma$ over a set $X$ is a function $\Gamma\colon X \to \Nat \cup \{ \infty \}$. A multiset over~$X$ is finite if $\Gamma(x) \in \Nat$ for each $x \in X$ and moreover $\Gamma$ only takes non-zero values on a finite set. Finite multisets can conveniently be described using the notation $[x_{1}; \dots; x_{n}]$. For example, $\Gamma \assign [1;1;2;1]$ is the unique multiset over $\Nat$ such that $\Gamma(1) = 3$, $\Gamma(2) = 1$, and $\Gamma(n) = 0$ otherwise. Multisets over $X$ are partially ordered by the componentwise order, also called the submultiset relation:
\begin{align*}
  \Gamma \leq \Delta \iff \Gamma(x) \leq \Delta(x) \text{ for each } x \in X.
\end{align*}
  Moreover, multisets over $X$ can be added componentwise:
\begin{align*}
  (\Gamma + \Delta)(x) \assign \Gamma(x) + \Delta(x).
\end{align*}
  This yields the partially ordered (additive) monoid $\Multi{X}$ of finite multisets over~$X$. If $\alg{M}$ is a (multiplicative) monoid, then $\Multi{\alg{M}}$ additionally inherits a multiplicative structure, which makes it a partially ordered semiring:
\begin{align*}
  [a_{1}; \dots; a_{m}] \cdot [b_{1}; \dots; b_{n}] \assign [a_{1} \cdot b_{1}; \dots; a_{1} \cdot b_{n}; \dots; a_{m} \cdot b_{1}; \dots; a_{n} \cdot b_{n}].
\end{align*}

  A natural approach to multiset-based consequence relations is now to replace sets by multisets in the definition of a consequence relation. This is indeed the approach taken by Cintula et al.~\cite{CGMP}, who develop this framework in detail and give it an abstract module-theoretic formulation in the spirit of Galatos and Tsinakis~\cite{GT}.

  Instead of working with relations $\Gamma \vdash \Delta$ where $\Gamma$ and $\Delta$ are sets, one takes $\Gamma$ and $\Delta$ to be multisets and replaces Reflexivity, Transitivity, Collection, and Structurality by their obvious multiset analogues. In particular, the subset relation $\Gamma \subseteq \Delta$ is replaced by the submultiset relation $\Gamma \leq \Delta$. Each set-based consequence relation $\vdash$ associated with a substructural logic with Weakening then yields a multiset-based consequence relation $\vdash_{\mathrm{m}}$ associated with the same logic:
\begin{align*}
  [\gamma_{1}; \dots; \gamma_{m}] \vdash_{\mathrm{m}} [\delta_{1}; \dots; \delta_{n}] \iff \gamma_{1} \cdot \ldots \cdot \gamma_{m} \vdash \delta_{1} \cdot \ldots  \cdot \delta_{n}.
\end{align*}
  This forces us to restrict to finite multisets in the definition of $\vdash_{\mathrm{m}}$ in order to for the right-hand side to be well-defined. In other words, we end up with a binary relation $\vdash_{\mathrm{m}}$ on the poset of finite multisets of formulas. The action of the monoid of substitutions $\alg{M}$ then extends to an action of~$\Multi{\alg{M}}$:
\begin{align*}
  [\sigma_{1}; \dots; \sigma_{n}] \ast [\gamma_{1}; \dots; \gamma_{m}] \assign [\sigma_{1} \ast \gamma_{1}; \dots; \sigma_{1} \ast \gamma_{m}; \dots; \sigma_{n} \ast \gamma_{1}; \dots; \sigma_{n} \ast \gamma_{m}].
\end{align*}
  At the abstract level, we therefore end up with a partially ordered semiring (to be thought of as an abstract counterpart of~$\Multi{\alg{M}}$) acting on a partially ordered monoid (to be thought of as an abstract counterpart of~$\Multi{Fm}$). Cintula et al.\ then recover analogues of the key results of Galatos and Tsinakis in this setting.

\subsection{Going substructural, take 2: from sets to sets of multisets}

  The central claim of present paper is that while it is tempting to generalize directly from consequence relations between sets to consequence relations between multisets, this strategy comes with a number of awkward features which can be avoided if one instead adopts an approach based on \emph{sets of multisets}. That is, our proposal is that instead of working with the set of finite multisets of formulas $\Multi{Fm}$, we should work with the set $\DM{Fm}$ of non-empty downsets of~$\Multi{Fm}$ with respect to the multiset order. The main technical obstacle of this proposal is figuring out how to extend the action of $\alg{M}$ on $Fm$ to an action of $\DM{\alg{M}}$ on $\DM{Fm}$, and indeed identifying the appropriate algebraic structure on $\DM{\alg{M}}$ and $\DM{Fm}$. Once this obstacle is dealt with, we obtain a smooth generalization of the framework of Galatos and Tsinakis~\cite{GT}, which is covered as a special (idempotent) case.

  While this approach looks more complicated than that of Cintula et al.~\cite{CGMP}, this appearance is deceiving. Cintula et al.\ do recover the correspondence of Galatos and Tsinakis between structural consequence relations and structural closure operators, but their structural closure operators on $\Multi{Fm}$ are not closure operators on $\Multi{Fm}$ in the usual sense. Rather, they are maps ${\Multi{Fm} \to \DM{Fm}}$. In other words, even in the framework of~\cite{CGMP} one cannot avoid dealing with $\DM{Fm}$ anyway.

  The above complication is not at all peculiar to the multiset setting, but rather it arises already in the set-based setting. In the set-based case, it is also entirely possible to \emph{a priori} restrict to the finitary case and accordingly to work with binary consequence relations on the set $\Power_{\omega}(Fm)$ of finite subsets of $Fm$. In that case, closure operators corresponding to such consequence relations again do not live on $\Power_{\omega}(Fm)$, but rather are maps $\Power_{\omega}(Fm) \to \Power(\Power_{\omega}(Fm))$, which can be simplified to maps $\Power_{\omega}(Fm) \to \Power(Fm)$. This is indeed a possible way to develop the theory of finitary consequence relations, and the work of Cintula et al. is in effect an extension of such an approach to the multiset-based setting. In our opinion, however, a smoother and more informative theory results from pursuing the path taken by Galatos and Tsinakis: to first develop the general framework and then to deal with finitary consequence relations as a special case.

  To add insult to injury, observe that the approach of Cintula et al. is not only intrinsically finitary, in fact their multiset-based counterpart $\vdash_{\mathrm{m}}$ of a set-based consequence relation $\vdash$ depends only on the formula--formula fragment of $\vdash$, i.e.\ it forgets everything about $\vdash$ except for the valid rules of the form $\varphi \vdash \psi$. In the absence of a conjunction connective such that $\{ \varphi, \psi \}$ is equivalent to $\{ \varphi \wedge \psi \}$ in $\vdash$, this can be a very destructive operation. In contrast, if $\Gamma$ and $\Delta$ are sets of multisets, one can encode the entire set-based consequence relation $\vdash$ (possibly even infinitary) into its multiset-based counterpart $\vdash_{\mathrm{m}}$ as follows:
\begin{align*}
  \Gamma \vdash_{\mathrm{m}} \Delta \iff \set{\gamma_{1} \cdot \ldots \cdot \gamma_{m}}{[\gamma_{1}; \dots; \gamma_{m}] \in \Gamma} \vdash \set{\delta_{1} \cdot \ldots \cdot \delta_{n}}{[\delta_{1}; \dots; \delta_{n}] \in \Delta}.
\end{align*}

  While our motivation is primarily to obtain a smooth algebraic theory of consequence relations, the more philosophically inclined logician also has reason to prefer our solution. As we already recalled, one of the selling points of substructural logics is that they allow us to precisely differentiate between two different ways of combining logical resources: a multiplicative or serial combination ($\varphi \cdot \psi$) and an additive or parallel combination ($\varphi \wedge \psi$). If the task at hand is to transfer this distinction to the setting of consequence relations, the framework based only on finite multisets of formulas does not fully deliver: it does not allow us to express additive combinations of premises. In contrast, the framework based on sets of finite multisets of formulas allows us to faithfully preserve the distinction between multiplicative and additive combination of resources at the inferential level.

  On the other hand, our framework does have one downside compared to that of Cintula et al., namely that our algebras of scalars are not definable by equations or by inequalities. Cintula et al.\ replace the quantales of scalars of Galatos and Tsinakis by partially ordered semirings satisfying some inequational conditions. In contrast, we replace them by what we call distributively generated generalized additive quantales with multiplication. While generalized additive quantales with multiplication are defined equationally, being distributively generated is not.

\subsection{Outline of the paper}

  In Section~\ref{sec: modules}, we introduce the key algebraic structures which we shall study in this paper, namely generalized additive quantales with multiplication and their modules. In the rest of the paper, we work relative to a prevariety of generalized quantales $\class{K}$. In each such prevariety we have the $\class{K}$-free generalized quantale $\FreeIn{\class{K}}{X}$ over a poset $X$. In Section~\ref{sec: free cdi} we show that if $\class{K}$ is the class of commutative dually integral generalized quantales, then $\FreeIn{\class{K}}{X}$ is the generalized quantale $\DM{X}$ of non-empty downsets of finitely generated multiupsets of $X$. In Section~\ref{sec: free quantales with multiplication} we show that if $\alg{M}$ is a partially ordered monoid, then $\FreeIn{\class{K}}{M}$ can be expanded to a generalized additive quantale with multiplication, and moreover each action of $\alg{M}$ on a generalized quantale $\alg{L}$ extends uniquely to a module action of $\FreeIn{\class{K}}{M}$ on $\alg{L}$. In Section~\ref{sec: nuclei} we establish the correspondence between structural consequence relations, structural nuclei, and module congruences. In Section~\ref{sec: projective} we describe the cyclic projective modules in our setting. Finally, in Section~\ref{sec: algebraizability} we prove an abstract form of Blok and Pigozzi's characterization of algebraizable logics in our substructural setting.

\section{Quantales and their modules}
\label{sec: modules}

  An \emph{(almost) complete join semilattice} is a poset where joins of all (non-empty) subsets exist. A homomorphism of such structures is a map which preserves all (non-empty) joins, and therefore also the partial order. Observe that complete join semilattices are precisely the almost complete join semilattices with a bottom element~$\bot$, and their homomorphisms are precisely the homomorphisms of almost complete join semilattices which preserve~$\bot$.

\begin{definition}
  A \emph{partially ordered monoid}, or a \emph{pomonoid} for short, is an ordered algebra, written either in multiplicative notation $\alg{M} \assign \langle M, \leq, \cdot, \1 \rangle$ or in additive notation $\alg{M} \assign \langle M, \leq, +, \0 \rangle$, consisting of a monoid equipped with a partial order such that monoidal operation is order-preserving in each coordinate. An additive pomonoid is \emph{dually integral} if $\0$ is its least element. It is \emph{idempotent} if $x + x = x$ for each $x$. It is \emph{commutative} if $x + y = y + x$ for each $x$ and $y$.
\end{definition}

  Additive pomonoids which are both dually integral and idempotent are always commutative. In fact, they are exactly the \emph{join semilattices with a zero} (bottom).

  The reader should keep in mind that additive and multiplicative structures will serve very different purposes in the following. In particular, quantales will be treated in additive notation. Addition is meant to abstract the sum of multisets, while multiplication is meant to abstract the composition of substitutions.

\begin{definition}
  A \emph{(generalized) quantale} is an algebra $\alg{Q} \assign \langle Q, \bigvee, +, \0 \rangle$ such that $\langle Q, \bigvee \rangle$ is an (almost) complete join semilattice, $\langle Q, \leq, +, \0 \rangle$ is a pomonoid with respect to the join semilattice order, and for all $x, y \in \alg{Q}$ and all (non-empty) families $x_{i}, y_{i} \in \alg{Q}$ with $i \in I$
\begin{align*}
  & x + \bigvee_{i \in I} y_{i} = \bigvee_{i \in I} (x + y_{i}), & & \bigvee_{i \in I} x_{i} + y = \bigvee_{i \in I} (x_{i} + y).
\end{align*}
  A (generalized) quantale is called commutative, dually integral, or idempotent if its additive pomonoid reduct is.
\end{definition}

  The reason for introducing generalized quantales is that some natural inequational axioms are incompatible with quantale structure.

\begin{example}
  A non-trivial dually integral generalized quantale cannot be a quantale (despite the fact that it is a complete join semilattice), since being a dually integral quantale requires that $x = x + \0 = x + \bot = \bot$ for each $x$. In particular, complete join semilattices are not quantales if we take $x + y \assign x \vee y$.
\end{example}

\begin{example}
  Complete join semilattices are precisely the generalized quantales where $x + y = x \vee y$. Equivalently, they are the idempotent (hence commutative) dually integral generalized quantales. Accordingly, the results of this paper form a non-idempotent generalization of the results of Galatos and Tsinakis~\cite{GT}.
\end{example}

  We now work towards a definition of modules where (generalized) additive quantales equipped with a multiplication act on (generalized) additive quantales. We start from the module action of a monoid on a (generalized) quantale.

\begin{definition}
  The \emph{monotone transformation pomonoid} $\Mon X$ of a poset~$X$ is the set of all order-preserving maps ${f\colon X \to X}$ with the componentwise order:
\begin{align*}
  f \leq^{\Mon X} g \text{ if and only if } f(a) \leq^{X} g(a) \text{ for all }a \in X,
\end{align*}
  and with the operations
\begin{align*}
  & f \cdot^{\Mon X} g \assign f \circ g, & & \1^{\Mon X} \assign \id_{X}.
\end{align*}
  The \emph{endomorphism pomonoid} $\End \alg{Q} \leq \alg{Q}^{Q}$ of a (generalized) quantale $\alg{Q}$ is the subpomonoid of $\alg{Q}^{Q}$ consisting of the endomorphisms of $\alg{Q}$, i.e.\ the homomorphisms of (generalized) quantales $\alg{Q} \to \alg{Q}$, with the monoidal structure inherited from $\alg{Q}^{Q}$, i.e.\ from the monoid of all functions $Q \to Q$ with the componentwise operations.
\end{definition}

  Observe that $\End \alg{Q} \leq \Mon Q$, where $Q$ is the poset reduct of $\alg{Q}$,

\begin{definition}
  An \emph{action} of a pomonoid $\alg{M}$ on a poset $X$ is a homomorphism of posets $\alg{M} \to \Mon X$. Equivalently, it is a map $\ast\colon \alg{M} \times X \to X$ which is order-preserving in both coordinates and moreover for all $a, b \in \alg{M}$ and $x \in X$
\begin{align*}
  & (a \cdot b) \ast x = a \ast (b \ast x), & & \1 \ast x = x.
\end{align*}
  An \emph{$\alg{M}$-poset} is a poset $X$ equipped with an action of $\alg{M}$.
\end{definition}

\begin{definition}
  An \emph{action} of a pomonoid $\alg{M}$ on a (generalized) quantale $\alg{Q}$ is a homomorphism of pomonoids $\alg{M} \to \End \alg{Q}$. Equivalently, it is a map $\ast\colon M \times Q \to Q$ which is an action of $\alg{M}$ on the poset reduct of $\alg{Q}$ such that for all $a \in \alg{M}$, $x, y \in \alg{Q}$, and each non-empty family $x_{i} \in \alg{Q}$ with $i \in I$
\begin{align*}
  & a \ast \bigvee_{i \in I} x_{i} = \bigvee_{i \in I} (a \ast x_{i}), & & a \ast (x + y) = a \ast x + a \ast y, & & a \ast \0 = \0.
\end{align*}
  An \emph{$\alg{M}$-act} is a (generalized) quantale $\alg{Q}$ equipped with an action of $\alg{M}$.
\end{definition}

  It will be convenient to introduce the notion of a \emph{(generalized) quantale term}, which is a (non-empty) set of additive monoidal terms, to be interpreted as a formal join. A (generalized) quantale term may involve an infinite set of variables, which requires us to admit possibly infinite tuples $\tuple{x}$, i.e.\ tuples indexed by ordinal numbers, as their arguments. The benefit of introducing (generalized) quantale terms is that we can now state the definition of an action of a pomonoid $\alg{M}$ on a (generalized) quantale $\alg{Q}$ more concisely: it is an order-preserving map ${\ast\colon M \times Q \to Q}$ such that for each (generalized) quantale term~$t$ and $a, b \in \alg{M}$, $x, \tuple{y} \in \alg{Q}$
\begin{align*}
  & (a \cdot b) \ast x = a \ast (b \ast x), & & \1 \ast x = x, & & a \ast t^{\alg{Q}}(\tuple{y}) = t^{\alg{Q}}(a \ast \tuple{y}).
\end{align*}
  We generally omit all reference to the length of our tuples and simply assume that whenever we use the notation $t^{\alg{Q}}(\tuple{a})$ that the tuple $\tuple{a}$ has an appropriate arity. Notation such as $\tuple{a} \ast x$ and $a \ast \tuple{x}$ is to be interpreted componentwise.

  In well-behaved cases, the pomonoid $\End \alg{Q}$ naturally inherits the structure of a (generalized) quantale from $\alg{Q}$.

\begin{example}
  Let $\alg{S}$ be a complete join semilattice, viewed as a generalized quantale $\alg{S} \assign \langle S, \bigvee, \vee, \bot \rangle$. Then $\End \alg{S}$ is a multiplicative generalized quantale
\begin{align*}
  \End \alg{S} \assign \langle \End \alg{S}, \bigvee, \vee, \bot, \circ, \id_{S} \rangle.
\end{align*}
\end{example}

  However, in other cases a (non-empty) join of endomorphisms need not be an endomorphism. This happens already in the case of commutative dually integral generalized quantales. The problem is that sums of join-preserving maps need not be join-preserving, and likewise joins of sum-preserving maps need not be sum-preserving. To overcome obstacle, we embed the pomonoid $\End \alg{Q}$ into the larger structure $\Mon Q \leq \alg{Q}^{Q}$ of order-preserving maps on $Q$.

\begin{definition}
  Let $\alg{Q}$ be a (generalized) quantale. Then $\Mon \alg{Q}$ is the expansion of the monotone transformation monoid $\langle \Mon Q, \cdot, 1 \rangle$ by the (generalized) additive quantale structure inherited by $\Mon Q$ as a (generalized) additive subquantale of~$\alg{Q}^{Q}$. 
\end{definition}

  $\Mon \alg{Q}$ is \emph{almost} a (generalized) quantale, except that in general $x \cdot \bigvee_{i \in I} y_{i}$ need not coincide with $\bigvee_{i \in I} (x \cdot y_{i})$.

\begin{definition}
  Let $\alg{Q}$ be a (generalized) quantale. Then $\Gen \alg{Q}$ denotes the subalgebra of $\Mon \alg{Q}$ generated by $\End \alg{Q}$. That is, $\Gen \alg{Q}$ is the set of all non-empty joins of sums of endomorphisms of $\alg{Q}$ and
\begin{align*}
  \Gen \alg{Q} \assign \langle \Sg^{\Mon \alg{Q}} \End \alg{Q}, \bigvee, +, \0, \circ, \id_{Q} \rangle,
\end{align*}
  where $\bigvee$, $+$, $\0$ are pointwise (non-empty) joins, sums, and zero, $\circ$ is functional composition, and $\id_{Q}$ is the identity function on~$Q$.
\end{definition}

  Of course, we do not wish to forget that inside of $\Gen \alg{Q}$ we have a subset of well-behaved elements $\End \alg{Q}$. As our expanded object of endomorphisms, we therefore take the two-sorted algebra
\begin{align*}
  \ExpEnd \alg{Q} \assign \langle \End \alg{Q}, \Gen \alg{Q}, \embedof{\alg{Q}} \rangle,
\end{align*}
  where $\embedof{\alg{Q}}\colon \End \alg{Q} \to \Gen \alg{Q}$ is the inclusion map. This results in a two-sorted algebra of the following kind.

\begin{definition}
  A \emph{(generalized) additive quantale with multiplication} is an algebra with two sorts $\A$ and $\Ad$ linked by a map $\iota\colon \Ad \to \A$ where
\begin{itemize}
\item $\Ad$ is a monoid,
\item $\A$ is a (generalized) quantale,
\item $\A$ additionally has the structure of a monoid $\langle A, \cdot, \1 \rangle$, and
\item $\embed\colon \Ad \to \A$ is a homomorphism of monoids,
\end{itemize}
  such that for all $\tuple{a}, b \in \A$ and each (non-empty) family $a_{i} \in \A$ with $i \in I$
\begin{align*}
  & \bigvee_{i \in I} a_{i} \cdot b = \bigvee_{i \in I} (a_{i} \cdot b), & & (a + b) \cdot c = (a \cdot c) + (b \cdot c), & & 0 \cdot a = 0,
\end{align*}
  and for all $d \in \Ad$, $\tuple{a} \in \A$, and each (non-empty) family $a_{i} \in \A$ with $i \in I$
\begin{align*}
  & \embed(d) \cdot \bigvee_{i \in I} a_{i} = \bigvee_{i \in I} (\embed(d) \cdot a_{i}), & & \embed(d) \cdot (a + b) = (\embed(d) \cdot a) + (\embed(d) \cdot b), & & \embed(d) \cdot 0 = 0.
\end{align*}
\end{definition}

  Equivalently, we can state these equations as: for each (generalized) quantale term~$t$ and all $\tuple{a} \in \A$, $b \in \A$, and $d \in \Ad$
\begin{align*}
  & t^{\A}(\tuple{a}) \cdot b = t^{\A}(\tuple{a} \cdot b), & & \embed(d) \cdot t^{\A}(\tuple{a}) = t^{\A}(\embed(d) \cdot \tuple{a}).
\end{align*}

  Abusing notation slightly, we shall generally use the name $\A$ to refer to the entire triple $\langle \Ad, \A, \embed \rangle$. We call $\Ad$ the \emph{monoid of distributive elements} of~$\A$. A (generalized) quantale with multiplication is said to be \emph{distributively generated} if $\A$ is generated as a (generalized) quantale by $\embed [\Ad]$. In particular, $\ExpEnd \alg{Q}$ is distributively generated by definition. 

  Instead of requiring that $\A$ be generated as a (generalized) quantale by $\embed [\Ad]$, we can equivalently require that the algebra $\langle A, \bigvee, +, \0, \cdot, \1 \rangle$ be generated by $\embed [\Ad]$.

\begin{lemma}
  A (generalized) additive quantale with multiplication $\langle \Ad, \A, \embed \rangle$ is distributively generated if and only if $\langle A, \bigvee, +, \0, \cdot, \1 \rangle$ is generated by $\embed[\Ad]$.
\end{lemma}

\begin{proof}
  The left-to-right implication is trivial. Conversely, it suffices to show that
\begin{align*}
  \set{t^{\A}(\tuple{a}) \in \A}{\text{$t$ is a (generalized) quantale term and } \tuple{a} \in \embed[\Ad]}
\end{align*}
  is closed under products and contains $\1$. But for any two such elements $t^{\A}(\tuple{a})$ and $u^{\A}(\tuple{b})$ 
\begin{align*}
  & t^{\A}(\tuple{a}) \cdot u^{\A}(\tuple{b}) = t^{\A}(\tuple{a} \cdot u^{\A}(\tuple{b})) = t^{\A}(u^{\A}(\tuple{a} \cdot \tuple{b})),
\end{align*}
  and clearly $\1^{\A} = \embedof{\A}(\1^{\Ad}) \in \embed[\Ad]$.
\end{proof}

  Homomorphisms of (generalized) additive quantales with multiplication are the ordinary homomorphisms of two-sorted algebras, i.e.\ such a homomorphism consists of a homomorphism of monoids $\hd\colon \Ad \to \Bd$ and a homomorphism of (generalized) quantales ${h\colon \A \to \B}$ such that $\embedof{\B} \circ \hd = h \circ \embedof{\A}$.

  We are now finally ready to define a module over a (generalized) additive quantale with multiplication~$\A$.

\begin{definition}
  An \emph{$\A$-module} is a (generalized) quantale $\alg{Q}$ equipped with a homomorphism $h\colon \A \to \ExpEnd \alg{Q}$ of (generalized) additive quantales with multiplication.
\end{definition}

  Homomorphisms $h\colon \A \to \ExpEnd \alg{Q}$ are in bijective correspondence with order-preserving maps $\ast\colon A \times Q \to Q$ such that for all $a, b \in \A$ and $x \in \alg{Q}$
\begin{align*}
  & (a \cdot b) \ast x = a \ast (b \ast x), & & \1 \ast x = x,
\end{align*}
  and for all $a, b \in \A$, $x \in \alg{Q}$, and each (non-empty) family $a_{i} \in \A$ with $i \in I$
\begin{align*}
  & (a + b) \ast x = a \ast x + b \ast x, & & \0 \ast x = \0, & & \bigvee_{i \in I} a_{i} \ast x = \bigvee_{i \in I} (a_{i} \ast x),
\end{align*}
  and for all $d \in \Ad$, $x, y \in \alg{Q}$, and each (non-empty) family $x_{i} \in \alg{Q}$ with $i \in I$
\begin{align*}
  & \embed(d) \ast (x + y) = \embed(d) \ast x + \embed(d) \ast y, & & \embed(d) \ast \0 = \0, & & \embed(d) \ast \bigvee_{i \in I} x_{i} = \bigvee_{i \in I} (\embed(d) \ast x_{i}).
\end{align*}
  Equivalently, we can state these equations more concisely: for each (generalized) quantale term~$t$ and all 
$\tuple{a} \in \A$, $x, \tuple{x} \in \alg{Q}$, and $d \in \Ad$
\begin{align*}
  & t^{\A}(\tuple{a}) \ast x = t^{\alg{Q}}(\tuple{a} \ast x), & & \iota(d) \ast t^{\alg{Q}}(\tuple{x}) = t^{\alg{Q}}(\iota(d) \ast \tuple{x}).
\end{align*}

\section{Free c.d.i.\ generalized quantales}
\label{sec: free cdi}

  We now have a definition of a module over a (generalized) additive quantale with multiplication. Apart from algebras of the form $\ExpEnd \alg{Q}$, however, we have not seen how to get our hands on a concrete additive quantale with multiplication.

  In the next section, we shall see that such modules can be obtained by adding multiplicative structure to \emph{free} (generalized) quantales over the poset reduct of a pomonoid, more precisely to (generalized) quantales free relative to a prevariety of (generalized) quantales. While the concrete shape of the free (generalized) quantale will be of no importance in the next section, nonetheless we now take the time in this section to describe these (generalized) quantales for a particular prevariety so that the reader can see how the work presented here subsumes as a special case a non-idempotent generalization of the work of Galatos and Tsinakis~\cite{GT}.

  Galatos and Tsinakis studied modules which were complete join semilattices, i.e.\ idempotent commutative dually integral generalized quantales. The closest natural non-idempotent generalization of their work is therefore to consider modules with the structure of commutative dually integral generalized quantales. We therefore devote the present section to describing this particular case in detail.

 Commutative dually integral pomonoids (generalized quantales) will be called \emph{c.d.i.\ pomonoids} (\emph{c.d.i.\ generalized quantales}) for short. (Recall that each c.d.i.\ quantale is trivial, hence the need to talk about generalized quantales rather than quantales.) Idempotent c.d.i.\ pomonoids are termwise equivalent to \emph{join semilattices with zero} if we interpret joins as sums and the bottom as zero.

\begin{definition}
  A \emph{prevariety of (generalized) quantales} is a class of (generalized) quantales closed under isomorphic images, subalgebras, and products. A \emph{prevariety of additive pomonoids} is a class of additive pomonoids closed under isomorphic images, ordered subalgebras, and products.
\end{definition}

\begin{definition}
  Let $\class{K}$ be a prevariety of (generalized) quantales. The \emph{$\class{K}$-free (generalized) quantale over a poset} $X$ is a (generalized) quantale $\FreeIn{\class{K}}{X}$ in $\class{K}$ with an order-preserving \emph{unit map} $\unit_{X}\colon X \to \FreeIn{\class{K}}{X}$ such that for each order-preserving map $h\colon X \to Q$ into the poset reduct $Q$ of an algebra $\alg{Q} \in \class{K}$ there is a unique homomorphism $h^{\sharp}\colon \FreeIn{\class{K}}{\alg{Q}} \to \alg{Q}$ which makes the following diagram commute:
\[
\begin{tikzcd}
 & \FreeIn{\class{K}} X \arrow[dr,"h^{\sharp}",dashed] \\
 & X \arrow[r,"h",swap] \arrow[u,"\unit_{X}"] & \alg{Q}
\end{tikzcd}
\]
  The \emph{$\class{K}$-free (generalized) quantale over an additive pomonoid $\alg{M}$}, denoted $\FreeIn{\class{K}}{\alg{M}}$, is defined in the same way, replacing posets by additive pomonoids.

  If $\class{K}$ is instead a prevariety of additive pomonoids, the \emph{$\class{K}$-free additive pomonoid over a poset} $X$ is also defined in the same way, \emph{mutatis mutandis}.
\end{definition}

  $\FreeIn{\class{K}}{X}$ exists for any poset $X$ and any prevariety $\class{K}$ of (generalized) quantales by the General Adjoint Functor Theorem~\cite[Theorem~4.6.3]{Riehl17}, since up to isomorphism the (generalized) quantales in $\class{K}$ generated by $X$ only form a set, rather than a proper class. $\FreeIn{\class{K}}{X}$ is unique up to a unique isomorphism commuting with $\unit_{X}$. It thus makes sense to talk about \emph{the} $\class{K}$-free (generalized) quantale over $X$.

  The following observation is immediate, since each order-preserving function $f\colon X \to X$ extends to an endomorphism $(\unit_{X} \circ f)^{\sharp}\colon \FreeIn{\class{K}}{X} \to \FreeIn{\class{K}}{X}$ which is the unique endomorphism $h$ of $\FreeIn{\class{K}}{X}$ with $h(\unit_{X}(x)) = \unit_{X}(f(x))$ for $x \in X$.

\begin{lemma} \label{lemma: posets to acts}
  Let $\class{K}$ be a prevariety of (generalized) quantales and $X$ be an $\alg{M}$-poset determined by $h\colon \alg{M} \to \Mon X$. Then $\FreeIn{\class{K}}{X}$ is an $\alg{M}$-act determined by the unique $h^{\sharp}\colon \alg{M} \to \End \FreeIn{\class{K}}{X}$ such that $h^{\sharp}(a)(\unit_{X}(x)) = \unit_{X}(h(a)(x))$ for $a \in \alg{M}$ and $x \in X$.
\end{lemma}

  The construction of the $\class{K}$-free (generalized) quantale over a poset can be split into two parts if $\class{K}$ is in fact a prevariety of additive pomonoids.

\begin{lemma} \label{lemma: concatenation}
  Let $\class{K}$ be a prevariety of additive pomonoids and $\class{K}_{\mathrm{q}}$ be the prevariety of (generalized) quantales whose additive pomonoid reduct lies in $\class{K}$. Then the $\class{K}_{\mathrm{q}}$-free (generalized) quantale over a poset $X$ is the $\class{K}_{\mathrm{q}}$-free (generalized) quantale over the $\class{K}$-free additive pomonoid over $X$ with the unit map $\unit_{\FreeIn{\class{K}}{X}} \circ \unit_{X}$.
\end{lemma}

\begin{proof}
  The composition of the forgetful functors from (generalized) quantales to additive pomonoids and from additive pomonoids to posets is the forgetful functor from (generalized) quantales to posets, and the composition of two adjunctions yields an adjunction.
\end{proof}

  In the case of c.d.i.\ generalized quantales, the more difficult construction will be the construction of the free c.d.i.\ additive pomonoids over a poset. The free c.d.i.\ generalized quantale over a c.d.i.\ additive pomonoid is then simply the generalized quantale of non-empty downsets. In the rest of this section, we describe these two constructions in concrete terms.

  We use the notation $\up P$ ($\down P$) for the upward (downward) closure of a subset~$P$ of some given poset, with $\up a \assign \up \{ a \}$ and $\down a \assign \down \{ a \}$. An upset is \emph{finitely generated} if it has the form $\up P$ for some finite set $P$. If an upset $U$ is finitely generated, then $U = \up (\min U)$, where $\min U$ is the (finite) set of minimal elements of~$U$.

\begin{example}
  The upsets of a poset $X$ ordered by inclusion form a join semilattice with zero with the operations
\begin{align*}
  & P + Q \assign P \cup Q, & & \0 \assign \emptyset.
\end{align*}
  The poset $X$ embeds into this join semilattice via the unit map $\unit_{X}\colon a \mapsto \up a$, and the subalgebra generated by elements of the form $\unit_{X}(a)$ is precisely the join semilattice with zero of \emph{finitely generated} upsets of $X$, denoted by $\UpFin{X}$. Each order-preserving map $h\colon X \to \alg{A}$, where $\alg{A}$ is a join semilattice with zero, then determines a homomorphism $h^{\sharp}\colon \UpFin{X} \to \alg{A}$ such that
\begin{align*}
  h^{\sharp}(P) \assign \bigvee h[\min P].
\end{align*}
  Because $\UpFin{X}$ is generated by elements of the form $\unit_{X}(a)$, this is in fact the unique homomorphism $h^{\sharp}\colon \UpFin{X} \to \alg{A}$ such that $h^{\sharp} \circ \unit_{X} = h$. Thus $\UpFin{X}$ is the free join semilattice with zero over the poset $X$.
\end{example}

  The upsets of $X$ are in bijective correspondence with order-preserving maps $f\colon X \to \Two$, where $\Two$ is the two-element join semilattice with zero, via the map sending each upset $P$ to its characteristic function $\chi_{P}\colon P \to \Two$. Under this bijection, the order and the operations of $\UpFin{X}$ are simply the pointwise order and operations inherited from $\Two$. The finitely generated upsets then correspond to maps $f$ such that $f^{-1} \{ 1 \}$ is a finitely generated upset of $X$.

  In the case of multiupsets, we replace the idempotent c.d.i.\ pomonoid $\Two$ by the non-idempotent c.d.i.\ pomonoid $\Nat$ of non-negative integers with the usual order and operations. This makes it slightly more complicated to describe the analogue of finitely generated upsets, but otherwise the construction of the pomonoid of multiupsets is identical to the idempotent case.

\begin{definition}
  A \emph{multiupset} over a poset~$X$ is an order-preserving map ${f\colon X \to \Nat}$. A multiupset $f$ is \emph{finitely generated} if (i) $f^{-1} [\up i]$ is a finitely generated upset of $X$ for each $i > 0$, and (ii) $f^{-1}[\up k] = \emptyset$ for some~$k \in \omega$.
\end{definition}

  The multiupsets over $X$ inherit from $\Nat$ the structure of a c.d.i.\ additive pomonoid with the pointwise order and operations. The neutral element of this pomonoid is the empty multiupset $[]$ such that $[](a) \assign 0$ for each $a \in X$.

  The finitely generated multiupsets over~$X$ form a subpomonoid $\Multi{X}$. The poset $X$ embeds into $\Multi{X}$ via the map $\unit_{X}\colon a \mapsto [a]$, where
\begin{align*}
  [a] (x) \assign \begin{cases} & 1 \text{ if } x \geq a, \\ & 0 \text{ otherwise.} \end{cases}
\end{align*}
  If $X$ is a set, i.e.\ a poset ordered by the equality relation, then the (finitely generated) multiupsets over $X$ are precisely the \emph{(finite) multisets} over $X$. We shall use the notation $\sum P$ for the sum of a finite subset $P$ of a c.d.i.\ pomonoid.

\begin{lemma}
  Each finitely generated multiupset $f$ over $X$ is a sum of elements of the form $[a]$, namely
\begin{align*}
   f = \sum_{i > 0} \left( \sum \min f^{-1} [\up i] \right).
\end{align*}
\end{lemma}

\begin{proof}
  The inner sums are finite by condition (i) in the definition of a finitely generated upset and the outer sum is finite by condition (ii), therefore the right-hand side of the equality is well-defined.

  We prove the equality by induction over the cardinality of $f$, defined as
\begin{align*}
  \card{f} \assign \sum_{i > 0} \card{f^{-1}[\up i]}.
\end{align*}
  Clearly each finitely generated multiupset has a well-defined finite cardinality. If $\card{f} = 0$, then $f = []$ and the equality holds, since indeed $[] = \sum \emptyset$. Now suppose that the equality holds for multiupsets $g$ with $\card{g} = n$, and consider a multiupset $f$ with $\card{f} = n+1$. Let $k \geq 1$ be the largest integer such that $f^{-1}[\up k]$ is non-empty. Then $f^{-1}[\up k] = \up \{ a_{1}, \dots, a_{i+1} \}$ for some distinct elements $a_{1}, \dots, a_{i+1} \in X$ with $i \in \omega$. Define the multiupset $g$ as follows:
\begin{align*}
  g (x) \assign \begin{cases} & f(x) - 1 \text{ if } x \geq a_{i+1}, \\ & f(x) \text{ otherwise.} \end{cases}
\end{align*}
  Then
\begin{align*}
  \min g^{-1}[\up j] = \begin{cases} \min f^{-1}[\up j] \text{ for } j < k, \\ (\min f^{-1}[\up j]) \setminus \{ a_{i+1} \} \text{ for } j = k, \\ \min f^{-1}[\up j] = \emptyset \text{ for } j > k. \end{cases}
\end{align*}
  Thus $g$ is finitely generated with $\card{g} = n$. Because $f = g + [a_{i+1}]$, the required equality for $f$ now follows immediately from the inductive hypothesis for $g$.
\end{proof}

  In view of the above lemma, a generic element of $\Multi{X}$ has the form
\begin{align*}
  [a_{1}, \dots, a_{n}] \assign [a_{1}] + \ldots + [a_{n}]
\end{align*}
  for some tuple $a_{1}, \dots, a_{n} \in X$, where the case $n \assign 0$ is to be interpreted as $[]$.

\begin{fact} \label{fact: free cdi pomonoid over poset}
  $\Multi{X}$ is the free c.d.i.\ pomonoid over the poset $X$ with the unit map $\unit_{X}\colon a \mapsto [a]$.
\end{fact}

\begin{proof}
  Consider a c.d.i.\ pomonoid $\alg{M}$ and an order-preserving map $h\colon X \to \alg{M}$. We define the map $h^{\sharp}\colon \Multi{X} \to \alg{M}$ as
\begin{align*}
  h^{\sharp}(f) \assign \sum_{i > 0} \sum h[\min f^{-1}[\up i]],
\end{align*}
  where the sums are taken in $\alg{M}$. Clearly $h^{\sharp}([]) = \0$, and moreover $h^{\sharp}(f + [a]) = h^{\sharp}(f) + h(a) = h^{\sharp}(f) + h^{\sharp}([a])$ for each $a \in X$. But by the previous lemma, $\Multi{X}$ is generated by elements of the form $\unit_{X}(a)$ for $a \in X$. It follows that $h^{\sharp}$ is a homomorphism and that it is the only homomorphism $h^{\sharp}\colon \Multi{X} \to \alg{M}$ such that $h^{\sharp} \circ \unit_{X} = h$.
\end{proof}

  The \emph{c.d.i.\ generalized quantale of non-empty downsets} of a c.d.i.\ pomonoid $\alg{M}$, denoted by $\Down{\alg{M}}$, is defined as the almost complete join semilattice of non-empty downsets of $M$ ordered by inclusion with the operations
\begin{align*}
  & P +^{\Down{\alg{M}}} Q \assign \down \set{p +^{\alg{M}} q}{p \in P \text{ and } q \in Q}, & & \0^{\Down{\alg{M}}} \assign \down \{ \0^{\alg{M}} \}.
\end{align*}
  The c.d.i.\ pomonoid $\alg{M}$ embeds into $\Down{\alg{M}}$ via the map $\unit_{\alg{M}}\colon a \mapsto \down \{ a \}$.

\begin{fact} \label{fact: free cdi quantale over pomonoid}
  $\Down{\alg{M}}$ is the free c.d.i.\ generalized quantale over the c.d.i.\ pomonoid $\alg{M}$ with the unit map $\unit_{\alg{M}}\colon a \mapsto \down \{ a \}$.
\end{fact}

\begin{proof}
  Consider a c.d.i.\ generalized quantale $\alg{N}$ and a homomorphism of c.d.i.\ pomonoids $h\colon \alg{M} \to \alg{N}$. We define the map $h^{\sharp}\colon \Down{\alg{M}} \to \alg{N}$ as
\begin{align*}
  h^{\sharp}(P) \assign \bigvee \set{h(p) \in \alg{N}}{p \in P}.
\end{align*}
  This map is a homomorphism of almost complete join semilattices and $h^{\sharp}(\down \{ a \}) = h(a)$. Moreover, $h^{\sharp}(\down \{ a \} + \down \{ b \}) = h^{\sharp}(\down \{ a + b \}) = h(a+b) = h(a) + h(b) = h^{\sharp}(\down \{ a \}) + h^{\sharp}(\down \{ b\})$ and $h^{\sharp}(\down \{ \0^{\alg{M}} \}) = h(\0^{\alg{M}}) = \0^{\alg{N}}$. But $\Down{\alg{M}}$ is generated as an almost complete join semilattice by elements of the form $\unit_{\alg{M}}(a)$ for $a \in \alg{M}$. It follows that $h^{\sharp}$ is a homomorphism and that it is the only homomorphism such that $h^{\sharp} \circ \unit_{\alg{M}} = h$.
\end{proof}

\begin{theorem} \label{thm: free cdi quantale}
  $\DM{X}$ is the free c.d.i.\ generalized quantale over a poset $X$ with the unit map $\unit_{X}\colon a \mapsto \down \{ [a] \}$.
\end{theorem}

\begin{proof}
  This follows immediately from Lemma~\ref{lemma: concatenation} and Facts~\ref{fact: free cdi pomonoid over poset} and \ref{fact: free cdi quantale over pomonoid}.
\end{proof}

\begin{example} \label{ex: dm fm}
  Let $\Fm$ be the absolutely free algebra (the algebra of formulas) in a given signature. The endomorphisms of $\Fm$ form the monoid of substitutions $\End \Fm$. Clearly the set of formulas $Fm$ is a discretely ordered $(\End \Fm)$-poset with respect to the action
\begin{align*}
  \sigma \ast \varphi \assign \sigma(\varphi).
\end{align*}
  By Lemma~\ref{lemma: posets to acts} the $(\End \Fm)$-poset $Fm$ extends uniquely to an $(\End \Fm)$-act $\DM{Fm}$ where
\begin{align*}
  \sigma \ast \down [\varphi] = \down [\sigma(\varphi)].
\end{align*}
  (Because $Fm$ is discretely ordered, we have $\down [\gamma] = \{ [\gamma], [] \}$ for every $\gamma \in Fm$.) More explicitly, the action of $\End \Fm$ on $Fm$ extends to an action on $\Multi{Fm}$:
\begin{align*}
  \sigma \ast [\varphi_{1}; \dots; \varphi_{n}] \assign [\sigma(\varphi_{1}); \dots; \sigma(\varphi_{n})],
\end{align*}
  which in turn extends to an action on $\DM{Fm}$:
\begin{align*}
  \sigma \ast P \assign \set{\sigma \ast p}{p \in P}.
\end{align*}
  Observe that $\sigma \ast P$ is indeed a downset of $\Multi{Fm}$ in the multiset order if $P$ is.

  Similarly, the set $Eq \assign Fm \times Fm$ of all equations in a given signature, written as $\varphi \equals \psi$, is a discretely ordered $(\End \Fm)$-poset with respect to the action
\begin{align*}
  \sigma \ast (\varphi \equals \psi) \assign \sigma(\varphi) \equals \sigma(\psi).
\end{align*}
  This $(\End \Fm)$-poset again extends uniquely to an $(\End \Fm)$-act $\DM{Eq}$ where
\begin{align*}
  \sigma \ast \down [\varphi \equals \psi] = \down [\sigma(\varphi) \equals \sigma(\psi)].
\end{align*}
\end{example}

\section{Free additive quantales with multiplication}
\label{sec: free quantales with multiplication}

  Throughout the following, let $\class{K}$ be a prevariety of (generalized) quantales. We shall generally suppress the subscript $\class{K}$ and assume some fixed choice of $\class{K}$.

  We now show how to expand the $\class{K}$-free (generalized) quantale $\Free{M}$ over the poset reduct~$M$ of a pomonoid $\alg{M}$ by a multiplication operation in order to obtain a (generalized) additive quantale with multiplication $\Free{\alg{M}}$. We then go on to prove that the categories of $\alg{M}$-acts and $\Free{\alg{M}}$-modules are isomorphic, provided that the unit map $\unit_{M}\colon M \to \Free{\alg{M}}$ is an embedding. In that case, categorical properties of $\alg{M}$-acts such as projectivity can be studied by means of $\Free{\alg{M}}$-modules.

  Let us start by considering an action $\ast$ of the monoid $\alg{M}$ on a (generalized) quantale~$\Q$. This action corresponds to a homomorphism of monoids
\begin{align*}
  & h\colon \alg{M} \to \End \Q \leq \Gen \Q, & & h(a)(x) \assign a \ast x.
\end{align*}
  (Recall that $\Gen \Q \leq \Q^{Q}$, so $\Q \in \class{K}$ implies $\Gen \Q \in \class{K}$.) The universal property of the free monoid $\Free{M}$ over the underlying set~$M$ of the monoid $\alg{M}$ ensures that the homomorphism of monoids~$h$ lifts to a homomorphism of (generalized) quantales
\begin{align*}
  h^{\sharp}\colon \Free{M} \to \Gen \Q \text{ such that } h^{\sharp} \circ \unit_{M} = h,
\end{align*}
  where $\unit_{M}\colon M \to \Free{M}$ is the unit map. The homomorphism $h^{\sharp}$ yields a map
\begin{align*}
  \ast^{\sharp}\colon \Free{M} \times Q \to Q \text{ such that } \unit_{M}(a) \ast^{\sharp} x = a \ast x \text{ for all } a \in \alg{M}.
\end{align*}
  Because $h^{\sharp}$ is a homomorphism, for all $\tuple{a} \in \Free{M}$ and $x \in \Q$
\begin{align*}
  t^{\Free{M}}(\tuple{a}) \ast^{\sharp} x = t^{\Q}(\tuple{a} \ast^{\sharp} x).
\end{align*}
  Because the codomain of $h$ is $\End \Q$, for all $a \in \alg{M}$ and $\tuple{x} \in \Q$
\begin{align*}
  \unit_{M}(a) \ast^{\sharp} t^{\Q}(\tuple{x}) = t^{\Q}(\unit_{M}(a) \ast^{\sharp} \tuple{x}).
\end{align*}
  To exhibit $\Q$ as a module over an expansion of the (generalized) quantale $\Free{M}$, it remains to define a multiplication and a multiplicative unit on $\Free{M}$, to show that this expansion turns $\Free{M}$ into a (generalized) additive quantale with multiplication, and that the above map $\ast^{\sharp}$ is an action with respect to this monoidal structure on $\Free{M}$.

  We first show that $\Free{M}$ has the structure of an $\alg{M}$-act. Each $a \in \alg{M}$ determines the map
\begin{align*}
  & \lambda_{a}\colon M \to M, & & \lambda_{a}\colon x \mapsto a \cdot x,
\end{align*}
  and therefore also the map
\begin{align*}
  & \unit_{M} \circ \lambda_{a}\colon M \to \Free{M}, & & \unit_{M} \circ \lambda_{a}\colon x \mapsto \unit_{M}(a \cdot x).
\end{align*}
  This latter map lifts to a homomorphism of (generalized) quantales
\begin{align*}
  (\unit_{M} \circ \lambda_{a})^{\sharp}\colon \Free{M} \to \Free{M} \text{ such that } (\unit_{M} \circ \lambda_{a})^{\sharp} \circ \unit_{M} = \unit_{M} \circ \lambda_{a}.
\end{align*}
   We therefore obtain a map
\begin{align*}
  & h\colon M \to \End \Free{M}, & & h\colon a \mapsto (\unit_{M} \circ \lambda_{a})^{\sharp}.
\end{align*}
  This map is a homomorphism of monoids, since
\begin{gather*}
  (\unit_{M} \circ \lambda_{\1})^{\sharp} = (\unit_{M} \circ \id_{M})^{\sharp} = \unit_{M}^{\sharp} = \id_{\Free{M}}, \\
  (\unit_{M} \circ \lambda_{a \cdot b})^{\sharp} = (\unit_{M} \circ \lambda_{a})^{\sharp} \circ (\unit_{M} \circ \lambda_{b})^{\sharp},
\end{gather*}
  where the second equation holds because
\begin{align*}
  (\unit_{M} \circ \lambda_{a})^{\sharp} \circ (\unit_{M} \circ \lambda_{b})^{\sharp} \circ \unit_{M} \kern-0.3pt = \kern-0.3pt (\unit_{M} \circ \lambda_{a})^{\sharp} \circ \unit_{M} \circ \lambda_{b} \kern-0.3pt = \kern-0.3pt \unit_{M} \circ \lambda_{a} \circ \lambda_{b} \kern-0.3pt = \kern-0.3pt \unit_{M} \circ \lambda_{a \cdot b}.
\end{align*}
  In other words, we have furnished $\Free{M}$ with the structure of an $\alg{M}$-act such that for $a, x \in \alg{M}$
\begin{align*}
  a \star \unit_{M}(x) = \unit_{M}(a \cdot x).
\end{align*}

  We now apply the construction $\ast \mapsto \ast^{\sharp}$ to the $\alg{M}$-act $\Free{M}$. This yields a map $\star^{\sharp}\colon \Free{M} \times \Free{M} \to \Free{M}$ such that for each (generalized) quantale term $t$ and for all $\tuple{a}, x, \tuple{x} \in \Free{M}$ and $a \in \alg{M}$
\begin{align*}
  & t^{\Free{M}}(\tuple{a}) \star^{\sharp} x = t^{\Free{M}}(\tuple{a} \star^{\sharp} x), & & \unit_{M}(a) \star^{\sharp} t^{\Free{M}}(\tuple{x}) = t^{\Free{M}}(\unit_{M}(a) \star^{\sharp} \tuple{x}),
\end{align*}
  and for all $a, b \in \alg{M}$
\begin{align*}
  \unit_{M}(a) \star^{\sharp} \unit_{M}(b) = \unit_{M}(a \cdot b).
\end{align*}
  These three equations uniquely determine the action $\star^{\sharp}$, since they allow us to reduce the computation of $t^{\Free{M}}(\unit_{M}(\tuple{a})) \star^{\sharp} u^{\Free{M}}(\unit_{M}(\tuple{b}))$ to the computation of products in $\alg{M}$. We can now finally define a multiplication on $\Free{M}$:
\begin{align*}
  & a \cdot^{\Free{\alg{M}}} b \assign a \star^{\sharp} b, & & \1^{\Free{\alg{M}}} \assign \unit_{M}(\1_{\alg{M}}).
\end{align*}
  Moreover, take $(\Free{\alg{M}})_{d} \assign \alg{M}$ with $\embed \assign \unit_{M}\colon M \to \Free{M}$. We claim that this defines a (generalized) additive quantale with multiplication $\Free{\alg{M}}$.

  Observe that there is a forgetful functor which takes a (generalized) additive quantale with multiplication $\alg{Q}$ to the pomonoid $\alg{Q}_{d}$. We can therefore talk about the free (generalized) additive quantale with multiplication over a pomonoid.

\begin{theorem} \label{thm: free aqm}
  Let $\alg{M}$ be a pomonoid. Then $\Free{\alg{M}}$ is the free (generalized) additive quantale with multiplication over $\alg{M}$. Moreover, it is distributively generated.
\end{theorem}

\begin{proof}
  Each element of $\Free{M}$ has the form $a \assign t(\tuple{a})$ for some (generalized) quantale term~$t$ and tuple $\tuple{a} \in \unit_{M}[\alg{M}]$. Let also $b \assign u(\tuple{b})$ and $c \assign v(\tuple{c}) \in \Free{M}$ for $\tuple{b}, \tuple{c} \in \unit_{M}[\alg{M}]$. The element $\1$ is a unit with respect to multiplication: 
\begin{align*}
  \1 \cdot a = \1 \cdot t^{\Free{M}}(\tuple{a}) = \unit_{M}(\1) \star^{\sharp} t^{\Free{M}}(\tuple{a}) = t^{\Free{M}}(\unit_{M}(\1) \star^{\sharp} \tuple{a})= t^{\Free{M}}(\tuple{a}) = a, \\
  a \cdot \1 = t^{\Free{M}}(\tuple{a}) \cdot \1 = t^{\Free{M}}(\tuple{a}) \star^{\sharp} \unit_{M}(\1) = t^{\Free{M}}(\tuple{a} \star^{\sharp} \unit_{M}(\1)) = t^{\Free{M}}(\tuple{a}) = a,
\end{align*}
  and multiplication in $\Free{M}$ is associative (we omit the superscripts here):
\begin{align*}
  (a \cdot b) \cdot c & = (t(\tuple{a}) \star^{\sharp} u(\tuple{b})) \star^{\sharp} v(\tuple{c}) 
  = (t(\tuple{a} \star^{\sharp} u(\tuple{b}))) \star^{\sharp} v(\tuple{c}) \\ &
  = (t(u(\tuple{a} \star^{\sharp} \tuple{b}))) \star^{\sharp} v(\tuple{c}) 
  = t(u(\tuple{a} \star^{\sharp} \tuple{b}) \star^{\sharp} v(\tuple{c})) \\ &
  = t(u((\tuple{a} \star^{\sharp} \tuple{b}) \star^{\sharp} v(\tuple{c}))) 
  = t(u(v((\tuple{a} \star^{\sharp} \tuple{b}) \star^{\sharp} \tuple{c}))) \\ &
  = t(u(v(\tuple{a} \star^{\sharp} (\tuple{b} \star^{\sharp} \tuple{c})))) 
  = t(u(\tuple{a} \star^{\sharp} (\tuple{b} \star^{\sharp} v(\tuple{c})))) \\ &
  = t(\tuple{a} \star^{\sharp} u((\tuple{b} \star^{\sharp} v(\tuple{c})))) 
  = t(\tuple{a} \star^{\sharp} (u(\tuple{b}) \star^{\sharp} v(\tuple{c}))) \\ &
  = t(\tuple{a}) \star^{\sharp} (u(\tuple{b}) \star^{\sharp} v(\tuple{c})) 
  = a \cdot (b \cdot c),
\end{align*}
  therefore $\Free{\alg{M}}$ is a monoid with respect to the given multiplication. Moreover, the map $\star^{\sharp}$ is an action, since $(a \cdot b) \star^{\sharp} c = (a \cdot b) \cdot c = a \cdot (b \cdot c) = a \star^{\sharp} (b \star^{\sharp} c)$ and $\1 \star^{\sharp} a = \1 \cdot a = a$. Finally, the map $\unit_{M}$ is a homomorphism of monoids:
\begin{align*}
  & \unit_{M}(a) \cdot \unit_{M}(b) = \unit_{M}(a) \star^{\sharp} \unit_{M}(b) = \unit_{M} (a \cdot b), & & \unit_{M}(\1) = \1.
\end{align*}
  This shows that $\Free{\alg{M}}$ is a (generalized) additive quantale with multiplication. It is distributively generated by definition. Now let $\A$ be a (generalized) additive quantale with multiplication and let $h\colon \alg{M} \to \Ad$ be a homomorphism of pomonoids. Then the homomorphism of pomonoids $\iota \circ h\colon \alg{M} \to \A$ yields a unique homomomorphism $(\iota \circ h)^{\sharp}\colon \Free{M} \to \A$ such that $(\iota \circ h)^{\sharp} \circ \unit_{M} = \iota \circ h$.
\end{proof}

  The forgetful functor assigning to each (generalized) additive quantale with multiplication $\A$ the pomonoid $\Ad$ induces a forgetful functor assigning to an $\A$-module with the action $\ast$ an $\Ad$-module with the action $\star$ such that $a \star x \assign \iota(a) \ast x$.

\begin{theorem} \label{thm: m and fm isomorphic}
  Suppose that the unit map $M \to \Free{M}$ is an order-embedding. Then the category of $\alg{M}$-acts and the category of $\Free{\alg{M}}$-modules are isomorphic via the forgetful functor. In particular, each $\alg{M}$-act can be uniquely expanded to an $\Free{\alg{M}}$-module with the module action $\ast^{\sharp}$.
\end{theorem}

\begin{proof}
  Consider an $\alg{M}$-act $\alg{L}$. We have already shown how to extend this action to a map $\ast^{\sharp}\colon \Free{M} \times L \to L$ which satisfies almost all of the required axioms. It remains to verify that $\ast^{\sharp}$ is an action: for $a \assign t(\tuple{a}), b \assign u(\tuple{b}) \in \Free{\alg{M}}$ and $c \in \alg{L}$
\begin{align*}
  (a \cdot b) \ast^{\sharp} c & = (t(\tuple{a}) \ast^{\sharp} u(\tuple{b})) \ast^{\sharp} c 
  = t(\tuple{a} \ast^{\sharp} u(\tuple{b})) \ast^{\sharp} c \\ &
  = t(u(\tuple{a} \ast^{\sharp} \tuple{b})) \ast^{\sharp} c 
  = t(u((\tuple{a} \ast^{\sharp} \tuple{b}) \ast^{\sharp} c)) \\ &
  = t(u(\tuple{a} \ast^{\sharp} (\tuple{b} \ast^{\sharp} c))) 
  = t(\tuple{a} \ast^{\sharp} u(\tuple{b} \ast^{\sharp} c)) \\ &
  = t(\tuple{a}) \ast^{\sharp} u(\tuple{b} \ast^{\sharp} c) 
  = t(\tuple{a}) \ast^{\sharp} (u(\tuple{b}) \ast^{\sharp} c) \\ &
  = a \ast^{\sharp} (b \ast^{\sharp} c)
\end{align*}
  and $\1 \ast^{\sharp} a = \unit_{M}(\1) \ast^{\sharp} t(\tuple{a}) = t(\unit_{M}(\1) \ast^{\sharp} \tuple{a}) = t(\tuple{a}) = a$. Thus each $\alg{M}$-act can be expanded to an $\Free{\alg{M}}$-module. Since $\Free{\alg{M}}$ is generated by $\alg{M}$ as a (generalized) quantale, this expansion is unique. Clearly the expansion construction and the forgetful restriction construction are mutually inverse.
\end{proof}

\begin{example} \label{example: dm m}
  Let $\alg{M}$ be a pomonoid and $\class{K}$ be the prevariety of c.d.i.\ generalized quantales. Then $\Free{M}$ is the generalized quantale $\DM{\alg{M}}$ of non-empty downsets of finitely generated multiupsets of $M$. It has the following $\alg{M}$-act structure:
\begin{align*}
  a \ast [x_{1}, \dots, x_{n}] & \assign [a \cdot x_{1}, \dots, a \cdot x_{n}], & & a \ast P \assign \bigcup \down \set{a \ast u}{u \in P},
\end{align*}
  for $a, x_{1}, \dots, x_{n} \in \alg{M}$ and $P \in \Free{\alg{M}}$. This extends to the structure of an $\Free{\alg{M}}$-module:
\begin{align*}
  [a_{1}, \dots, a_{n}] \ast P & \assign a_{1} \ast P + \ldots + a_{n} \ast P, & & P \ast Q \assign \bigcup \down \set{a \ast Q}{a \in P}.
\end{align*}
  Observe that this does not coincide with the na\"{\i}ve elementwise action. Rather, when computing $P \ast Q$, one first reduces this to computing actions of the form $a \ast Q$, each of which then reduces to actions of the form $a \ast x$.
\end{example}

\begin{example} \label{ex: dm fm 2}
  In Example~\ref{ex: dm fm} we saw that $\DM{Fm}$ and $\DM{Eq}$ are $\alg{M}$-acts for $\alg{M} \assign (\End \Fm)$. Theorems~\ref{thm: free aqm} and~\ref{thm: m and fm isomorphic} allow us to extend these to $\DM{\alg{M}}$-modules. Each element of $\DM{\alg{M}}$ is a set of finite multisets of elements of $\alg{M}$. We first extend the action of $\alg{M}$ to $\Multi{\alg{M}}$:
\begin{align*}
  [\sigma_{1}; \dots; \sigma_{n}] \ast^{\sharp} P \assign (\sigma_{1} \ast P) + \dots + (\sigma_{n} \ast P).
\end{align*}
  This in turn extends to an action of $\DM{\alg{M}}$:
\begin{align*}
  \Sigma \ast^{\sharp} P \assign \down \set{[\sigma_{1}; \dots; \sigma_{n}] \ast^{\sharp} P}{[\sigma_{1}; \dots; \sigma_{n}] \in \Sigma}.
\end{align*}
\end{example}

\section{Nuclei and homomorphic images}
\label{sec: nuclei}

  Before we describe the cyclic projective modules over a distributively generated (generalized) additive quantale with multiplication in the next section, we need to set up a correspondence between the homomorphic images of a given module $\Q$ and well-behaved closure operators on $\Q$, which we call structural nuclei. These are, in addition, in correspondence with additive consequence relations.

  The homomorphic images of a (generalized) quantale $\Q$ can of course be described up to isomorphism in terms of the \emph{congruences} on $\Q$, i.e.\ equivalence relations $\theta$ on $Q$ such that
\begin{enumerate}[(i)]
\item $\pair{a}{b}, \pair{c}{d} \in \theta$ implies $\pair{a+c}{b+d} \in \theta$, and
\item $\pair{a_{i}}{b_{i}} \in \theta$ for $i \in I$ (with $I$ non-empty) implies $\pair{\bigvee_{i \in I} a_{i}}{\bigvee_{i \in I} b_{i}} \in \theta$.
\end{enumerate}
  These form a lattice $\Con \Q$. Equivalently, congruences can be described by additive consequence relations on $\Q$.

\begin{definition}
  An \emph{additive consequence relation} on a (generalized) quantale $\Q$ is a binary relation $\vdash$ on $Q$ such that for all $x, y, z \in \Q$
\begin{enumerate}[(i)]
\item if $x \geq y$, then $x \vdash y$,
\item if $x \vdash y$ and $y \vdash z$, then $x \vdash z$,
\item $x \vdash \bigvee \set{y \in \Q}{x \vdash y}$,
\item if $x \vdash y$, then $x + z \vdash y + z$ and $z + x \vdash z + y$.
\end{enumerate}
\end{definition}

  Observe that for each (generalized) quantale term $t$ if $\tuple{x} \vdash \tuple{y}$, i.e.\ if $x_{i} \vdash y_{i}$ for each index $i$, then $t^{\Q}(\tuple{x}) \vdash t^{\Q}(\tuple{y})$.

\begin{fact}
  Let $\Q$ be a (generalized) quantale. Then the lattice $\Con \Q$ and the lattice of additive consequence relations on $\Q$ ordered by inclusion are isomorphic via the maps ${{\vdash} \mapsto \theta_{\vdash}}$ and $\theta \mapsto {\vdash_{\theta}}$ given by:
\begin{align*}
  & \pair{x}{y} \in \theta_{\vdash} \iff x \vdash y \text{ and } y \vdash x, & & x \vdash_{\theta} y \iff \pair{x \vee y}{x} \in \theta.
\end{align*}
  Equivalently, $x \vdash_{\theta} y$ if and only if $y / \theta \leq x / \theta$ in $\Q / \theta$.
\end{fact}

\begin{definition}
  A \emph{nucleus} on a (generalized) quantale $\Q$ is a map $\gamma\colon Q \to Q$ such that
\begin{enumerate}[(i)]
\item $\gamma$ is order-preserving,
\item $\gamma$ is \emph{expansive}: $x \leq \gamma(x)$ for each $x \in \Q$,
\item $\gamma$ is \emph{idempotent}: $\gamma(\gamma(x)) = \gamma(x)$ for each $x \in \Q$,
\item $\gamma(x) + \gamma(y) \leq \gamma(x+y)$ for each $x, y \in \Q$.
\end{enumerate}
\end{definition}

  The first three conditions simply state that $\gamma$ is a closure operator on $\Q$. The last condition extends to the inequality $t^{\Q}(\gamma(\tuple{x})) \leq \gamma(t^{\Q}(\tuple{x}))$ for all terms~$t$ and all tuples $\tuple{x}$ of elements in $\Q$.

  Each additive consequence relation $\vdash$ on $\Q$ defines a nucleus $\gamma_{\vdash}$ on $\Q$:
\begin{align*}
  \gamma_{\vdash} (x) \assign \bigvee \set{y \in \Q}{x \vdash y}.
\end{align*}
  Conversely, a nucleus $\gamma$ on $\Q$ defines an additive consequence relation $\vdash_{\gamma}$ on~$\Q$:
\begin{align*}
  x \vdash_{\gamma} y \iff y \leq \gamma(x).
\end{align*}
  It is not immediately obvious that the poset of all nuclei on $\Q$ ordered pointwise forms a lattice, but it follows from the next proposition, since the additive consequence relations on $\Q$ form a lattice.

\begin{proposition} \label{prop: nuclei and consequence}
  The lattice of all nuclei on a (generalized) quantale $\Q$ ordered pointwise and the lattice of all additive consequence relations on $\Q$ ordered by inclusion are isomorphic via the maps ${\vdash} \mapsto {\gamma_{\vdash}}$ and ${\gamma} \mapsto {\vdash_{\gamma}}$.
\end{proposition}

\begin{proof}
  Let $\gamma$ be a nucleus. Then (i) $x \leq y$ implies $x \leq \gamma(y)$, (ii) $x \leq \gamma(y)$ and $y \leq \gamma(z)$ imply $x \leq \gamma(z)$, (iii) $\bigvee \set{y \in \Q}{y \leq \gamma(x)} \leq \gamma(x)$, and (iv) $y \leq \gamma(x)$ implies $y + z \leq \gamma(x) + z \leq \gamma(x) + \gamma(z) \leq \gamma(x+z)$, and likewise $z + y \leq \gamma(z+x)$. Thus $\vdash_{\gamma}$ is an additive consequence relation.

  Conversely, let $\vdash$ be an additive consequence relation on $\Q$. Then $\gamma_{\vdash}$ is an order-preserving map because $x \vdash z$ implies $y \vdash z$ for $x \leq y$ by (i) and (ii), it is an extensive map by (i), and it is idempotent: $\gamma_{\vdash}(\gamma_{\vdash}(x)) \leq \gamma_{\vdash}(x)$ because $\gamma_{\vdash} (x) \vdash y$ implies $x \vdash y$, thanks to the fact that $x \vdash \gamma_{\vdash} (x)$ by (iii). Finally, $\gamma_{\vdash}(x_{1}) + \gamma_{\vdash}(x_{2}) = \bigvee \set{y_{1} \in \Q}{x_{1} \vdash y_{1}} + \bigvee \set{y_{2} \in \Q}{x_{2} \vdash y_{2}} = \bigvee \set{y_{1} + y_{2}}{{x_{1} \vdash y_{1}} \text{ and } x_{2} \vdash y_{2}} \leq \bigvee \set{y \in \Q}{x_{1} + x_{2} \vdash y}$, since these joins are non-empty and $x_{1} \vdash y_{1}$ and $x_{2} \vdash y_{2}$ imply that $x_{1} + x_{2} \vdash y_{1} + y_{2}$. Thus $\gamma_{\vdash}$ is a nucleus.

  The two maps are clearly order-preserving. They are also mutually inverse:
\begin{align*}
  x \vdash_{\gamma_{\vdash}} y \iff y \leq \gamma_{\vdash}(x) \iff y \leq \bigvee \set{z \in \Q}{x \vdash z} \iff x \vdash y
\end{align*}
  since $x \vdash \bigvee \set{z \in \Q}{x \vdash z}$, and conversely
\begin{align*}
  x \leq \gamma_{\vdash_{\gamma}} (y) \iff x \leq \bigvee \set{z \in \Q}{y \vdash_{\gamma} z} \iff x \leq \bigvee \set{z \in \Q}{z \leq \gamma (y)},
\end{align*}
  but $\bigvee \set{z \in \Q}{z \leq \gamma (y)} = \gamma(y)$, so $x \leq \gamma_{\vdash_{\gamma}} (y) \iff x \leq \gamma(y)$ for each $x \in \Q$, and thus $\gamma_{\vdash_{\gamma}}(y) = \gamma(y)$.
\end{proof}

  Given a nucleus $\gamma$ on a (generalized) quantale $\Q$, we define the (generalized) quantale $\Q_{\gamma}$ over the set $\gamma[Q]$ as
\begin{align*}
  & \bigvee_{\gamma} X \assign \gamma \left( \bigvee X \right), & & x +^{\Q_{\gamma}} y \assign \gamma(x +^{\Q} y), & & 0^{\Q_{\gamma}} \assign \gamma(0^{\Q}).
\end{align*}
  The (generalized) quantale $\Q_{\gamma}$ is isomorphic to the quotient $\Q / \theta$ where $\theta$ is the congruence corresponding to $\vdash_{\gamma}$.

\begin{fact}
  Let $\Q$ be a (generalized) quantale. Then $\Con \Q$ and the lattice of all nuclei on $\Q$ are isomorphic via the maps $\theta \mapsto \gamma_{\theta}$ and $\gamma \mapsto \theta_{\gamma}$ given by:
\begin{align*}
  & \gamma_{\theta}(x) \assign \bigvee [x]_{\theta}, & & \pair{x}{y} \in \theta_{\gamma} \iff \gamma(x) = \gamma(y).
\end{align*}
  The (generalized) quantales $\Q_{\gamma}$ and $\Q / \theta_{\gamma}$ are isomorphic.
\end{fact}

  A congruence $\theta$ on an $\A$-module $\Q$ is said to be \emph{structural with respect to $a \in \A$} if for all $x, y \in \Q$
\begin{align*}
  \text{if $\pair{x}{y} \in \theta$, then $\pair{a \ast x}{a \ast y} \in \theta$.}
\end{align*}
  A \emph{structural} congruence is structural with respect to each $a \in \A$. The following lemma shows that it suffices to verify structurality with respect to $\iota[\Ad]$, provided that $\A$ is distributively generated.

\begin{lemma} \label{lemma: distributively generated congruence}
  If $\A$ is distributively generated and the congruence $\theta$ on an $\A$-module $\Q$ is structural with respect to each $a \in \iota[\Ad]$, then $\theta$ is structural.
\end{lemma}

\begin{proof}
  For each (generalized) quantale term $t$ and all tuples ${\tuple{a} \in \Ad}$: if $\pair{x}{y} \in \theta$, then $\pair{\tuple{a} \ast x}{\tuple{a} \ast y} \in \theta$, so $\pair{t^{\Q}(\tuple{a}) \ast x}{t^{\Q}(\tuple{a}) \ast y} = \pair{t^{\Q}(\tuple{a} \ast x)}{t^{\Q}(\tuple{a} \ast y)} \in \theta$.
\end{proof}

  Similarly, an additive consequence relation $\vdash$ on an $\A$-module $\Q$ is said to be \emph{structural with respect to $a \in \A$} if for all $x, y \in \Q$
\begin{align*}
  \text{if $x \vdash y$, then $a \ast x \vdash a \ast y$.}
\end{align*}
  A \emph{structural} additive consequence relation on $\Q$ is structural with respect to each $a \in \A$. The following lemma again shows that it suffices to verify structurality with respect to $\iota[\Ad]$, provided that $\A$ is distributively generated.

\begin{lemma} \label{lemma: distributively generated consequence}
  If $\A$ is distributively generated and the additive consequence relation $\vdash$ on an $\A$-module $\Q$ is structural with respect to each $a \in \iota[\Ad$, then $\vdash$ is structural.
\end{lemma}

\begin{proof}
  This holds because for each generalized quantale term $t$ and all tuples ${\tuple{a} \in \Ad}$: if $x \vdash y$, then $\tuple{a} \ast x \vdash \tuple{a} \ast y$, so $t^{\Q}(\tuple{a} \ast x) \vdash t^{\Q}(\tuple{a} \ast y)$ and $t^{\Q}(\tuple{a}) \ast x = t^{\Q}(\tuple{a} \ast x) \vdash t^{\Q}(\tuple{a} \ast y) = t^{\Q}(\tuple{a}) \ast y$.
\end{proof}

  Finally, a nucleus $\gamma$ on an $\A$-module $\Q$ is said to be \emph{structural with respect to $a \in \A$} if for all $x \in \Q$
\begin{align*}
  a \ast \gamma(x) \leq \gamma(a \ast x).
\end{align*}
  A \emph{structural nucleus} on $\Q$ is a nucleus structural with respect to each $a \in \A$. Again, it suffices to verify structurality with respect to $\iota[\Ad]$, provided that $\A$ is distributively generated.

\begin{lemma} \label{lemma: distributively generated nucleus}
  If $\A$ is distributively generated and the nucleus $\gamma$ on an $\A$-module $\Q$ is structural with respect to each $a \in \iota[\Ad]$, then $\gamma$ is a structural nucleus.
\end{lemma}

\begin{proof}
  If $\tuple{a} \ast \gamma(x) \leq \gamma(\tuple{a} \ast x)$, then for each function symbol $f$
\begin{align*}
  f^{\Q}(\tuple{a}) \ast \gamma(x) = f^{\Q}(\tuple{a} \ast \gamma(x)) \leq f^{\Q}(\gamma(\tuple{a} \ast x)) \leq \gamma(f^{\Q}(\tuple{a} \ast x)).
\end{align*}
  Likewise, if $a_{i} \ast x \leq \gamma(a_{i} \ast x)$ for each $a_{i}$ with $i \in I$, then we have
\begin{align*}
  \bigvee_{i \in I} a_{i} \ast \gamma(x) \leq \bigvee_{i \in I} \gamma(a_{i} \ast x) \leq \gamma \left( \bigvee_{i \in I} \gamma(a_{i} \ast x) \right) = \gamma \left( \bigvee_{i \in I} a_{i} \ast x \right),
\end{align*}
  so $a \ast \gamma(x) \leq \gamma (a \ast x)$ for $a \assign \bigvee_{i \in I} a_{i}$.
\end{proof}

\begin{lemma}
  Let $\gamma$ be a structural nucleus on an $\A$-module $\Q$. Then $\Q_{\gamma}$ is an $\A$-module with respect to the action
\begin{align*}
  a \ast_{\gamma} x \assign \gamma(a \ast x).
\end{align*}
\end{lemma}

\begin{proof}
  For each term $t$ and each tuple $\tuple{a} \in \A$
\begin{align*}
  t^{\A}(\tuple{a}) \ast_{\gamma} x = \gamma(t^{\A}(\tuple{a}) \ast x) = \gamma(t^{\Q}(\tuple{a} \ast x)) = \gamma(t^{\Q}(\gamma(\tuple{a} \ast x))) = t^{\Q_{\gamma}}(\tuple{a} \ast_{\gamma} x),
\end{align*}
  and for each term $t$, $d \in \Ad$, and $\tuple{x} \in \Q_{\gamma}$
\begin{align*}
  \iota(d) \ast_{\gamma} t^{\Q_{\gamma}}(\tuple{x}) & =  \gamma(\iota(d) \ast t^{\Q_{\gamma}}(\tuple{x})) = \gamma(\iota(d) \ast \gamma(t^{\Q}(\tuple{x}))) = \gamma(\iota(d) \ast t^{\Q}(\tuple{x})) \\ & = \gamma(t^{\Q}(\iota(d) \ast \tuple{x})) = \gamma(t^{\Q}(\gamma(\iota(d) \ast \tuple{x}))) = t^{\Q_{\gamma}}(\iota(d) \ast_{\gamma} \tuple{x}).
\end{align*}
  Finally, for each $x \in \Q_{\gamma}$ we have $\1 \ast_{\gamma} x = \gamma(\1 \ast x) = \gamma (x) = x$ and
\begin{align*}
  & (a \cdot b) \ast_{\gamma} x = \gamma((a \cdot b) \ast x) = \gamma(a \ast (b \ast x)) = \gamma(a \ast \gamma(b \ast x)) = a \ast_{\gamma} (b \ast_{\gamma} x). \qedhere
\end{align*}
\end{proof}

  The above correspondence between congruences, additive consequence relations, and nuclei extends to a correspondence between structural congruences, structural additive consequence relations, and structural nuclei. To this end, it suffices to show that these isomorphisms preserve structurality in both directions.

\begin{theorem} \label{thm: nuclei and consequence}
  Let $\Q$ be an $\A$-module where $\A$ is distributively generated. Then:
\begin{enumerate}[(i)]
\item If a congruence $\theta$ on $\Q$ is structural, then so is $\vdash_{\theta}$.
\item If a nucleus $\gamma$ on $\Q$ is structural, then so is $\vdash_{\gamma}$.
\item If an additive consequence relation $\vdash$ on $\Q$ is structural, then so are $\theta_{\vdash}$ and~$\gamma_{\vdash}$.
\end{enumerate}
\end{theorem}

\begin{proof}
  Given Lemmas~\ref{lemma: distributively generated congruence}, \ref{lemma: distributively generated consequence} and~\ref{lemma: distributively generated nucleus}, it suffices to show that these maps preserve structurality with respect to each $a \in \Ad$.

  If $\gamma$ is structural with respect to $a$, then $x \leq \gamma(y)$ implies $a \ast x \leq a \ast \gamma(y) \leq \gamma(a \ast y)$, so $\vdash_{\gamma}$ is structural with respect to $a$. Conversely, if $\vdash$ is structural with respect to $a$, then 
\begin{align*}
  a \ast \bigvee \set{y \in \Q}{x \vdash y} = \bigvee \set{a \ast y}{x \vdash y} \leq \bigvee \set{z}{a \ast x \vdash z},
\end{align*}
  where the inequality holds because $x \vdash y$ implies $a \ast x \vdash a \ast y$, so for each $y \in \Q$ we can take $z \assign a \ast y$. Thus $a \ast \gamma_{\vdash} (x) \leq \gamma_{\vdash}(a \ast x)$.

  If $\theta$ is structural with respect to $a$, then $x \vdash_{\theta} y$ implies $\pair{x \vee y}{x} \in \theta$, so $\pair{(a \ast x) \vee (a \ast y)}{a \ast x} = \pair{a \ast (x \vee y)}{a \ast x} \in \theta$ and $a \ast x \vdash_{\theta} a \ast y$. Conversely, if $\vdash$ is structural with respect to $a$, then $\pair{x}{y} \in \theta_{\vdash}$ implies $x \vdash y \vdash x$, so $a \ast x \vdash a \ast y \vdash a \ast x$ and $\pair{a \ast x}{a \ast y} \in \theta_{\vdash}$.
\end{proof}

\section{Cyclic projective modules}
\label{sec: projective}

  The goal of this section is to extend the description of the cyclic projective $\A$-modules in~\cite[Theorem~5.7]{GT} beyond the idempotent case treated in~\cite{GT}.

\begin{definition}
  The $\A$-module $\P$ is \emph{projective} if for each surjective homomorphism of $\A$-modules ${g\colon \Q \onto \R}$ every homomorphism of $\A$-modules $h\colon \P \to \R$ lifts to some homomorphism $h^{\sharp}\colon \P \to \Q$ such that $h = g \circ h^{\sharp}$:
\[
\begin{tikzcd}
 & & \Q \arrow[d,"g",->>] \\
 & \P \arrow[r,"h"] \arrow[ur,"h^{\sharp}",dashed] & \R
\end{tikzcd}
\]
\end{definition}

  Throughout this section, $\A$ will be a \emph{distributively generated} (generalized) additive quantale with multiplication and $\Q$ is an $\A$-module. Elements of $\A$ will be denoted by $a, b, c$, while elements of $\Q$ will be denoted by $x, y, z$ or $u, v, w$.

\begin{definition}
  Given a pomonoid $\alg{M}$, an $\alg{M}$-poset $X$, and an element $u \in X$, let
\begin{align*}
  M \ast u \assign \set{a \ast u \in X}{a \in \alg{M}}.
\end{align*}
  An $\alg{M}$-poset $X$ is \emph{cyclic}, or more explicitly \emph{$u$-cyclic}, if $X = M \ast u$ for some $u \in X$. An $\alg{M}$-act $\P$ is \emph{cyclic}, or more explicitly \emph{$u$-cyclic}, if $\P$ is generated as a (generalized) quantale by $M \ast u$ for some $u \in \P$. An $\A$-module $\Q$ is \emph{cyclic}, or more explicitly \emph{$u$-cyclic}, if $Q = A \ast u$ for some $u \in \Q$.
\end{definition}

  Given $u \in \Q$, the set $A \ast u$ forms a submodule $\A \ast u$ of $\Q$. Clearly $\Q$ is $u$-cyclic if and only if $\Q = \A \ast u$.

  If $\A$ is a quantale, then the (left) action of $\A$ on $\Q$ has a \emph{residual}. That is, for each $x, y \in \Q$ there is some $y \sast x \in \A$ such that for all $a \in \A$
\begin{align*}
  a \ast x \leq_{\Q} y \iff a \leq_{\A} y \sast x,
\end{align*}
  namely $y \sast x \assign \bigvee \set{b}{b \ast x \leq_{\Q} y}$. If $\A$ is only a \emph{generalized} quantale, the element $y \sast x$ exists if and only if there is some $a \in \A$ such that $a \ast x \leq y$.

\begin{definition}
  An element $u \in \Q$ is a \emph{dividing element} if $x \sast u$ exists for each $x \in \Q$, i.e.\ if for each $y \in \Q$ there is some $a \in \A$ such that $a \ast u \leq y$.
\end{definition}

  Clearly if $\Q$ is a dually integral generalized quantale, then each element of $\Q$ is a dividing element, since $\0 \ast x = \0 \leq y$ for each $x, y \in \Q$.

\begin{lemma}
  $\Q$ is $u$-cyclic if and only if $u$ is a dividing element and $(x \sast u) \ast u = x$ for each $x \in \Q$.
\end{lemma}

\begin{proof}
  The right-to-left implication is immediate, as is the fact that if $\Q$ is $u$-cyclic, then $u$ is a dividing element. Finally, if $u$ is a dividing element, then $(x \sast u) \ast u \leq x$ and $a \leq (a \ast u) \sast u$ for each $a \in \A$ and $x \in \Q$, so $x = a \ast u$ implies that $x = a \ast u \leq ((a \ast u) \sast u) \ast u = (x \sast u) \ast u \leq x$ and $(x \sast u) \ast u = x$.
\end{proof}

  The (generalized) additive quantale with multiplication $\A$ may itself be viewed as an $\A$-module if we take $a \ast x \assign a \cdot x$.

\begin{lemma}
  Let $\gamma$ be a structural nucleus on the $\A$-module $\A$. Then $\A_{\gamma}$ is a cyclic $\A$-module with a cyclic generator $\gamma(\1)$.
\end{lemma}

\begin{proof}
  The homomorphic image of a $u$-cyclic $\A$-module with respect to a surjective homomorphism $h$ is an $h(u)$-cyclic $\A$-module. But $\A$ is $\1$-cyclic and $\gamma$ is a surjective homomorphism of modules from $\A$ onto $\A_{\gamma}$.
\end{proof}

\begin{lemma} \label{lemma: gamma u}
  Let $\Q$ be an $\A$-module with a dividing element $u \in \Q$. Then:
\begin{enumerate}[(i)]
\item The map $\gamma_{u}\colon a \mapsto (a \ast u) \sast u$ is a structural nucleus on $\A$.
\item $\A \ast u$ is isomorphic to $\A_{\gamma_{u}}$ via the maps $x \mapsto x \sast u$ and $a \mapsto a \ast u$.
\end{enumerate}
  In particular, each $u$-cyclic $\A$-module $\Q$ is isomorphic to $\A_{\gamma_{u}}$.
\end{lemma}

\begin{proof}
  Proving that $\gamma_{u}$ is a closure operator is straightforward, and $a \cdot \gamma_{u} (b) = a \cdot ((b \ast u) \sast u) \leq (a \ast (b \ast u)) \sast u = ((a \cdot b) \ast u) \sast u = \gamma_{u}(a \cdot b)$. Moreover, for each function symbol $f$ and $\tuple{a} \in \A$ we have $f^{\A}(\gamma_{u}(\tuple{a})) = f^{\A}((\tuple{a} \ast u) \sast u) \leq (f^{\A}(\tuple{a}) \ast u) \sast u = \gamma_{u}(f^{\A}(\tuple{a}))$, since $f^{\A}((\tuple{a} \ast u) \sast u) \ast u = f^{\A}(((\tuple{a} \ast u) \sast u) \ast u) = f^{\A}(\tuple{a} \ast u) = f^{\A}(\tuple{a}) \ast u$. Thus $\gamma_{u}$ is a structural nucleus.

  The two maps in (ii) are homomorphisms of (generalized) quantales: for each (generalized) quantale term $t$ and $\tuple{x} \assign \tuple{a} \ast u$ we have $t^{\A_{\gamma_{u}}}(\tuple{x} \sast u) = \gamma_{u}(t^{\A}(\tuple{x} \sast u)) = (t^{\A}((\tuple{a} \ast u) \sast u) \ast u) \sast u = t^{\Q}(((\tuple{a} \ast u) \sast u) \ast u) \sast u = t^{\Q}(\tuple{a} \ast u) \sast u = t^{\Q}(\tuple{x}) \sast u$. They are mutually inverse maps: for $x = a \ast u$ we have $(x \sast u) \ast u = ((a \ast u) \sast u) \ast u = a \ast u = x$, and for $a = \gamma_{u}(b)$ we have $(a \ast u) \sast u = (((b \ast u) \sast u) \ast u) \sast u = (b \ast u) \sast u = \gamma_{u}(b) = a$. Finally, they are compatible with the action: for $a \assign \gamma_{u}(b) = (b \ast u) \sast u$ we have $(a \ast u) \sast u = (((b \ast u) \sast u) \ast u) \sast u = (b \ast u) \sast u = a$.
\end{proof}

  The last two lemmas yield the following description of cyclic $\A$-modules.

\begin{theorem} \label{thm: cyclic modules}
  An $\A$-module $\Q$ is cyclic if and only if it is isomorphic to an $\A$-module of the form $\A_{\gamma}$ for some structural nucleus $\gamma$ on $\A$.
\end{theorem}

  Cyclic projective $\A$-modules also admit a description analogous to that of~\cite{GT}.

\begin{lemma} \label{lemma: gamma equals gamma u}
  Let $\gamma$ be a structural nucleus on $\A$ and let $u \in \A$ be a dividing element. Then the following are equivalent:
\begin{enumerate}[(i)]
\item $\gamma = \gamma_{u}$ and $u \cdot u = u$.
\item $\gamma(u) = \gamma(\1)$ and $\gamma(a) \cdot u = a \cdot u$.
\end{enumerate}
\end{lemma}

\begin{proof}
  The proof of Lemma~5.6 of~\cite{GT} carries over word for word.
\end{proof}

\begin{theorem} \label{thm: cyclic projective modules}
  The following are equivalent:
\begin{enumerate}[(i)]
\item $\Q$ is a cyclic projective $\A$-module.
\item $\Q$ is isomorphic to $\A \cdot u \iso \A_{\gamma_{u}}$ for some dividing idempotent $u \in \A$.
\item $\Q$ is $v$-cyclic for some $v \in \Q$ and $\gamma_{v} = \gamma_{u}$ for some dividing idempotent $u \in \A$.
\item $\Q$ is $v$-cyclic for some $v \in \Q$ and there is some dividing element $u \in \A$ such that $\gamma_{v}(u) = \gamma_{v}(\1)$ and $\gamma_{v}(a) \cdot u = a \cdot u$ for all $a \in \A$.
\item $\Q$ is $v$-cyclic for some $v \in \Q$ and there is some dividing element $u \in \A$ such that $u \ast v = v$ and $((a \ast v) \sast v) \cdot u = a \cdot u$ for all $a \in \A$.
\end{enumerate}
  Moreover, the elements $u \in \A$ and $v \in \Q$ can be taken to be the same in all conditions in which they appear.
\end{theorem}

\begin{proof}
  The equivalence (ii) $\Leftrightarrow$ (iii) follows immediately from Lemma~\ref{lemma: gamma u}. The equivalence (iii) $\Leftrightarrow$ (iv) is Lemma~\ref{lemma: gamma equals gamma u}. The equivalence (iv) $\Leftrightarrow$ (v) holds because $u \ast v = v$ implies that $\gamma_{v}(\1) = (\1 \ast v) \sast v = (\1 \ast (u \ast v)) \sast v = (u \ast v) \sast v = \gamma_{v}(u)$, and $\gamma_{v}(u) = \gamma_{v}(\1)$ implies that $u \ast v = ((u \ast v) \sast v) \ast v = \gamma_{v}(u) \ast v = \gamma_{v}(\1) \ast v = ((\1 \ast v) \sast v) \ast v = v$.

  To prove the implication (i) $\Rightarrow$ (ii), let $\Q$ be a cyclic projective $\A$-module. By Theorem~\ref{thm: cyclic modules}, up to isomorphism the module $\Q$ has the form $\A_{\gamma}$ for some structural nucleus $\gamma$ on $\A$. But $\gamma\colon \A \to \A_{\gamma}$ is a surjective homomorphism of modules, so by projectivity there is a homomorphism of modules $h\colon \A_{\gamma} \to \A$ such that $\gamma \circ h = \id_{\A_{\gamma}}$. Let $u \assign h(\gamma(\1))$. Then $u \cdot u = h(\gamma(\1)) \cdot h(\gamma(\1)) = h(h(\gamma(\1)) \ast_{\gamma} \gamma(\1)) = h(\gamma(h(\gamma(\1)) \cdot \gamma(\1))) = h(\gamma(h(\gamma(\1)) \cdot \1)) = h(\gamma(h(\gamma(\1)))) = h(\gamma(\1)) = u$. Moreover, $h(\gamma(a)) = h(\gamma(a \cdot \1)) = h(\gamma(a \cdot \gamma(\1))) = h(a \ast_{\gamma} \gamma(\1)) = a \ast h(\gamma(\1)) = a \cdot u$, so $h[\A_{\gamma}] = \A \cdot u$. But $h$ is an order embedding, since $\gamma \circ h = \id_{\A_{\gamma}}$, hence $\A_{\gamma}$ and $\A \cdot u$ are isomorphic modules.

  To prove the implication (ii) $\Rightarrow$ (i), we show that a module of the form $\A \cdot u$ for some idempotent $u \in \A$ is projective. Consider a surjective homomorphism of $\A$-modules $h\colon \P \onto \Q$ and a homomorphism $f\colon \A \cdot u \to \Q$. Because $h$ is surjective, there is some $v \in \P$ such that $h(v) = f(u)$. Consider the map $g\colon A \cdot u \to P$ defined as $g\colon a \cdot u \mapsto (a \cdot u) \ast v$. Clearly $(h \circ g) (a \cdot u) = h((a \cdot u) \ast v) = (a \cdot u) \ast h(v) = (a \cdot u) \ast f(u) = f((a \cdot u) \ast u) = f(a \cdot (u \cdot u)) = f(a \cdot u)$, hence $h \circ g = f$. Moreover, $g$ is a homomorphism of $\A$-modules: for each term $t$ and tuple $\tuple{a} \in \A$ we have $g(t^{\A}(\tuple{a}) \cdot u) = (t^{A}(\tuple{a}) \cdot u) \ast v = t^{\A}(\tuple{a} \cdot u) \ast v= t^{\P}((\tuple{a} \cdot u) \ast v) = t^{\P}(g(\tuple{a} \cdot u))$.
\end{proof}

  We can obtain a description of the cyclic projective $\alg{M}$-posets and $\alg{M}$-acts for a given pomonoid $\alg{M}$ from the above theorem. Recall from Lemma~\ref{lemma: posets to acts} that each $\alg{M}$-poset $X$ extends to an $\alg{M}$-act $\Free{X}$.

\begin{lemma} \label{lemma: cyclic act to module}
  An $\alg{M}$-poset $X$ is $u$-cyclic if and only if the $\alg{M}$-act $\Free{X}$ is $\unit_{X}(u)$-cyclic. An $\alg{M}$-act $\alg{L}$ is $u$-cyclic if and only if the $\Free{\alg{M}}$-module $\alg{L}$ is $u$-cyclic.
\end{lemma}

\begin{proof}
  If the $\alg{M}$-poset $X$ is $u$-cyclic, then the $\alg{M}$-act $\Free{X}$ is $\unit_{X}(u)$-cyclic because $\Free{X}$ is generated as a (generalized) quantale by $\unit_{X}[X]$ and in $\Free{X}$ we have $a \ast \unit_{X}(u) = \unit_{X}(a \ast u)$. Conversely, if the $\alg{M}$-act $\Free{X}$ is $\unit_{X}(u)$-cyclic, then it is generated as a (generalized) quantale by $M \ast \unit_{X}(u)$, but this implies that each of the free generators $\unit_{X}(x)$ with $x \in X$ has the form $\unit_{X}(x) = a \ast \unit_{X}(u)$ for some $a \in \alg{M}$. But then $x = a \ast u$ in the $\alg{M}$-poset $X$, proving that this $\alg{M}$-poset is cyclic.

  An $\alg{M}$-act $\alg{L}$ is $u$-cyclic if and only if each $x \in \alg{L}$ has the form $x = t^{\alg{L}}(\tuple{a} \ast u)$ for some (generalized) quantale term $t$ and some $\tuple{a} \in \alg{M}$. But this holds if and only if $x = t^{\Free{\alg{M}}}(\tuple{a}) \ast u$ in the expansion of $\alg{L}$ to an $\Free{\alg{M}}$-module.
\end{proof}

\begin{theorem}
  Let $\alg{M}$ be a pomonoid such that the unit map $M \to \Free{M}$ is an order-embedding. Then the cyclic projective $\alg{M}$-acts are precisely those isomorphic to $\Free{\alg{M}} \cdot u$ for some dividing idempotent $u \in \Free{\alg{M}}$.
\end{theorem}

\begin{proof}
  This follows immediately from the isomorphism between the category of $\alg{M}$-acts and the category of $\Free{\alg{M}}$-module (Theorem~\ref{thm: m and fm isomorphic}), which preserves and reflects surjectivity of morphisms and therefore also the projectivity of objects.
\end{proof}

\begin{example} \label{ex: dm fm 3}
  In Example~\ref{ex: dm fm 2} we saw that $\DM{Fm}$ and $\DM{Eq}$ are $\DM{\alg{M}}$-modules for $\alg{M} \assign (\End \Fm)$. They are both cyclic $\DM{\alg{M}}$-modules: the $\alg{M}$-posets $Fm$ and $Eq$ are cyclic, the cyclic generators being respectively any variable $x \in \Fm$ and any equation $x \equals y$ with distinct variables $x, y \in Fm$, therefore by Lemma~\ref{lemma: cyclic act to module} the $\alg{M}$-acts $\DM{Fm}$ and $\DM{Eq}$, hence also the $\DM{\alg{M}}$-modules $\DM{Fm}$ and $\DM{Eq}$, are cyclic, the cyclic generators being respectively $\down [x] = \{ [x], [] \}$ and $\down [x \equals y] = \{ [x \equals y], [] \}$.

  We now use Theorem~\ref{thm: cyclic projective modules} to prove that these two cyclic $\DM{\alg{M}}$-modules are projective. Consider first the case of $\DM{Fm}$. Take $\sigma_{x} \in (\End \Fm)$ to be the substitution with $\sigma_{x}(y) \assign x$ for each variable $y$, and let $u \assign \down [\sigma_{x}]$ and $v \assign \down [x]$. Because $\DM{\alg{M}}$ is dually integral, $u$ is a dividing element. Clearly $u \ast v = \down[\sigma_{x}] \ast \down [x] = \down [\sigma_{x} \ast x] = \down [x] = v$.

  It remains to prove that $a \cdot u = ((a \ast v) \sast v) \cdot u$ for all $a \in \DM{\alg{M}}$. The inequality $a \cdot u \leq ((a \ast v) \sast v) \cdot u$ always holds. To prove the inequality
\begin{align*}
  ((a \ast v) \sast v) \cdot u \leq a \cdot u,
\end{align*}
  observe that each element of $\DM{\alg{M}}$ is a join of elements of the form $[\sigma_{1}; \dots; \sigma_{n}]$. Since products distribute over joins on the left, it therefore suffices to show that
\begin{align*}
  \down [\sigma_{1}; \dots; \sigma_{n}] \leq (a \ast v) \sast v \implies \down [\sigma_{1}; \ldots; \sigma_{n}] \cdot \down [\sigma_{x}] \leq a \cdot \down [\sigma_{x}],
\end{align*}
  i.e.\ that
\begin{align*}
  \down [\sigma_{1}(x); \dots; \sigma_{n}(x)] \leq a \ast \down [x] \implies \down [\sigma_{1}; \ldots; \sigma_{n}] \cdot \down [\sigma_{x}] \leq a \cdot \down [\sigma_{x}].
\end{align*}
  If $\down [\sigma_{1}(x); \dots; \sigma_{n}(x)] \leq a \ast \down [x]$ for some $a \in \DM{\alg{M}}$, then there is multiset $[\tau_{1}; \dots; \tau_{k}] \in a$ such that $[\sigma_{1}(x); \dots; \sigma_{n}(x)] \leq [\tau_{1}(x); \dots; \tau_{k}(x)]$. But then $[\sigma_{1}; \dots; \sigma_{n}] \cdot [\sigma_{x}] = [\sigma_{1} \circ \sigma_{x}; \dots; \sigma_{n} \circ \sigma_{x}] \leq [\tau_{1} \circ \sigma_{x}; \dots; \tau_{k} \circ \sigma_{x}]$, so $\down [\sigma_{1}; \dots; \sigma_{n}] \cdot \down [\sigma_{x}] \leq \down [\tau_{1}; \dots; \tau_{k}] \cdot \down [\sigma_{x}] \leq a \cdot \down [\sigma_{x}]$, proving that $\DM{Fm}$ is indeed a cyclic projective $\DM{\alg{M}}$-module.

  To prove that $\DM{Eq}$ is a cyclic projective $\DM{\alg{M}}$-module, we similarly take $u \assign \down [\sigma_{x,y}]$ and $v \assign \down [x]$, where $\sigma_{x, y}$ is a substitution such that $\sigma_{x, y}(x) = x$, $\sigma_{x, y}(y) = y$, and $\sigma_{x, y}(z) \in \{ x, y \}$ for variables $z$ other than $x$, $y$. The rest of the argument is entirely analogous.
\end{example}

\section{Equivalence of substructural consequence relations}
\label{sec: algebraizability}

  It remains to put the framework developed above to logical use. As in~\cite{GT}, we prove an abstract analogue of Blok and Pigozzi's characterization of algebraizable logics~\cite[Theorem~3.7(ii)]{BP89}.

\begin{theorem}
  Let $\alg{A}$ be a (generalized) additive quantale with multiplication. Then an $\alg{A}$-module $\alg{P}$ is projective if and only if each embedding of modules $f\colon \alg{P}_{\gamma} \into \alg{Q}_{\delta}$ for each pair of structural nuclei $\gamma$ on $\alg{P}$ and $\delta$ on $\alg{Q}$ (where $\alg{Q}$ is an $\alg{A}$-module) is induced by a homomorphism of modules $\tau\colon \alg{P} \to \alg{Q}$ in the sense that $f \circ \gamma = \delta \circ \tau$.
\end{theorem}

\begin{proof}
  The left-to-right implication is immediate: it suffices to apply the definition of projectivity to the homomorphism $f \circ \gamma\colon \alg{P} \to \alg{Q}_{\delta}$ and the surjective homomorphism $\delta\colon \alg{Q} \to \alg{Q}_{\delta}$. Conversely, each homomorphism $h\colon \alg{P} \to \alg{Q}_{\delta}$ decomposes into a surjective homomorphism, without loss of generality of the form $\gamma\colon \alg{P} \to \alg{P}_{\gamma}$ for some structural nucleus $\gamma$ on $\alg{P}$, followed by an embedding $f\colon \alg{P}_{\gamma} \into \alg{Q}_{\delta}$. By assumption, there is a homomorphism $\tau\colon \alg{P} \to \alg{Q}$ such that $f \circ \gamma = \delta \circ \tau$, so $h = f \circ \gamma = \delta \circ h^{\sharp}$ for $h^{\sharp} \assign \tau$, proving that $\alg{P}$ is projective.
\end{proof}

\begin{corollary} \label{cor: algebraizability}
  Let $\alg{P}$ and $\alg{Q}$ be projective $\alg{A}$-modules, and let $\gamma$ and $\delta$ be structural nuclei on $\alg{P}$ and $\alg{Q}$, respectively. Then each isomorphism of modules given by $f\colon \alg{P}_{\gamma} \to \alg{Q}_{\delta}$ and $g\colon \alg{Q}_{\delta} \to \alg{P}_{\gamma}$ is induced by a pair of homomorphisms of modules $\tau\colon \alg{P} \to \alg{Q}$ and $\rho\colon \alg{Q} \to \alg{P}$ in the sense that $f \circ \gamma = \delta \circ \tau$ and $g \circ \delta = \gamma \circ \rho$:
\[
\begin{tikzcd}
  & \alg{P} \arrow[r,"\tau",shift left,dashed] \arrow[d,"\gamma",->>] & \alg{Q} \arrow[l,"\rho",shift left,dashed] \arrow[d,"\delta",->>] \\
  & \alg{P}_{\gamma} \arrow[r,"f",shift left] & \alg{Q}_{\delta} \arrow[l,"g",shift left]
\end{tikzcd}
\]
\end{corollary}

  The following theorem now restates the above corollary in the language of consequence relations.

\begin{theorem} \label{thm: algebraizability}
  Let $\alg{P}$ and $\alg{Q}$ be projective $\alg{A}$-modules, and let $\gamma$ and $\delta$ be structural nuclei on $\alg{P}$ and $\alg{Q}$, respectively. Then $\alg{P}_{\gamma}$ and $\alg{Q}_{\gamma}$ are isomorphic $\alg{A}$-modules if and only if there are homomorphisms of modules $\tau\colon \alg{P} \to \alg{Q}$ and $\rho\colon \alg{Q} \to \alg{P}$ such that for all $\Gamma, \Delta \in \alg{P}$ and $E, F \in \alg{Q}$
\begin{align*}
  \makebox[10pt]{$\Gamma$} \makebox[15pt]{$\vdash_{\gamma}$} \makebox[7.5pt]{$\Delta$} & \iff \tau(\makebox[7.6pt]{$\Gamma$}) \vdash_{\delta} \tau(\Delta), & & \makebox[10pt]{$E$} \makebox[15pt]{$\vdash_{\delta}$} (\tau \circ \rho)(\makebox[7.6pt]{$E$}) \vdash_{\delta} E, \\
  \makebox[10pt]{$E$} \makebox[15pt]{$\vdash_{\delta}$} \makebox[7.5pt]{$F$} & \iff \rho(\makebox[7.6pt]{$E$}) \vdash_{\gamma} \rho(F), & & \makebox[10pt]{$\Gamma$} \makebox[15pt]{$\vdash_{\gamma}$} (\rho \circ \tau)(\makebox[7.6pt]{$\Gamma$}) \vdash_{\gamma} \Gamma.
\end{align*}
  Moreover, $\tau$ and $\rho$ satisfy the first line above if and only if they satisfy the second.
\end{theorem}

\begin{proof}
  If $f\colon \alg{P}_{\gamma} \to \alg{Q}_{\delta}$ and $g\colon \alg{Q}_{\delta} \to \alg{P}_{\gamma}$ form an isomorphism of modules, then by Corollary~\ref{cor: algebraizability} there are homomorphisms of modules$\tau\colon \alg{P} \to \alg{Q}$ and $\rho\colon \alg{Q} \to \alg{P}$ such that $f \circ \gamma = \delta \circ \tau$ and $g \circ \delta = \gamma \circ \rho$. The first equation implies that for all $\Gamma, \Delta \in \alg{P}$
\begin{align*}
  \Gamma \vdash_{\gamma} \Delta & \iff \gamma(\Delta) \leq \gamma(\Gamma) \\
  & \iff f(\gamma(\Delta)) \leq f(\gamma(\Gamma)) \\
  & \iff \delta(\tau(\Delta)) \leq \delta(\tau(\Gamma)) \\
  & \iff \tau(\Gamma) \vdash_{\delta} \tau(\Delta).
\end{align*}
  Likewise, the equation $g \circ \delta = \gamma \circ \rho$ implies that for all $E, F \in \alg{Q}$
\begin{align*}
  E \vdash_{\delta} F & \iff \rho(E) \vdash_{\gamma} \rho(F).
\end{align*}
  Finally, the fact that $f$ and $g$ are mutually inverse maps implies that
\begin{align*}
  (\gamma \circ \rho \circ \tau)(\Gamma) = (g \circ \delta \circ \tau)(\Gamma) = (g \circ f \circ \gamma)(\Gamma) = \gamma(\Gamma).
\end{align*}
  Therefore $\Gamma \vdash_{\gamma} \rho(\tau(\Gamma)) \vdash_{\gamma} \Gamma$, and likewise $E \vdash_{\delta} \tau(\rho(E)) \vdash_{\delta} E$.

  Conversely, if the required maps $\tau$ and $\rho$ exist, let $f \assign \delta \circ \tau\colon \alg{P}_{\gamma} \to \alg{Q}_{\delta}$ and $g \assign \gamma \circ \rho\colon \alg{Q}_{\delta} \to \alg{P}_{\gamma}$. It suffices to show that these are mutually inverse maps. But for all $\Gamma \in \alg{P}_{\gamma}$ and $\Delta \in \alg{P}$
\begin{align*}
  \Delta \leq (g \circ f)(\Gamma) & \iff \Delta \leq (\gamma \circ \rho \circ \delta \circ \tau)(\Gamma) \\
  & \iff (\rho \circ \delta \circ \tau)(\Gamma) \vdash_{\gamma} \Delta \\
  & \iff (\rho \circ \delta \circ \tau)(\Gamma) \vdash_{\gamma} (\rho \circ \tau)(\Delta) \\
  & \iff (\delta \circ \tau)(\Gamma) \vdash_{\delta} \tau(\Delta) \\
  & \iff \tau(\Gamma) \vdash_{\delta} \tau(\Delta) \\
  & \iff \Gamma \vdash_{\gamma} \Delta \\
  & \iff \Delta \leq \gamma(\Gamma) = \Gamma.
\end{align*}
  Thus $(g \circ f)(\Gamma) = \Gamma$. Likewise, $(f \circ g)(E) = E$ for all $E \in \alg{Q}_{\delta}$.

  The claim that the first line is equivalent to the second is entirely standard, but let us repeat it for the sake of being self-contained. If the first line holds, then
\begin{align*}
  \rho(E) \vdash_{\gamma} \rho(F) \iff (\tau \circ \gamma)(E) \vdash_{\delta} (\tau \circ \rho)(F) \iff E \vdash_{\delta} F,
\end{align*}
  and $\tau(\Gamma) \vdash_{\delta} (\tau \circ \rho)(\tau(\Gamma)) \vdash_{\delta} \tau(\Gamma)$ implies that $\Gamma \vdash_{\gamma} (\rho \circ \tau)(\Gamma) \vdash_{\gamma} \Gamma$. The converse implication is analogous.
\end{proof}

\begin{example}
  In Example~\ref{ex: dm fm 3} we proved that $\DM{Fm}$ and $\DM{Eq}$ are cyclic projective $\DM{\End \Fm}$-modules, so the above theorem in particular applies to the case of $\alg{P} \assign \DM{Fm}$ and $\alg{Q} \assign \DM{Eq}$. The cyclic generators of these modules are $\down [x]$ and $\down [x \equals y]$, respectively, where $x$ and $y$ are distinct variables in~$Fm$. The homomorphisms of modules $\tau\colon \alg{P} \to \alg{Q}$ and $\rho\colon \alg{Q} \to \alg{P}$ are thus uniquely determined by the following elements of $\Q$ and $\P$:
\begin{align*}
  & \boldsymbol{\tau}(x) \assign \tau(\down [x]), & & \boldsymbol{\rho}(x, y) \assign \rho(\down [x \equals y]),
\end{align*}
  which are respectively a downset of finite multisets of equations in the variable $x$ and a downset of finite multisets of formulas in the variables $x, y$. Conversely, any choice of $\boldsymbol{\tau}(x)$ and $\boldsymbol{\rho}(x, y)$ determines a pair of homomorphisms $\tau\colon \P \to \Q$ and $\rho\colon \Q \to \P$: the value of $\tau(\Gamma)$ for an arbitrary downset $\Gamma$ of finite multisets is computed as the union of the values of $\tau(\down [\gamma_{1}; \dots; \gamma_{n}])$ for all finite multisets $[\gamma_{1}; \dots; \gamma_{n}] \in \Gamma$, which are in turn computed as $\boldsymbol{\tau}(\gamma_{1}) + \dots + \boldsymbol{\tau}(\gamma_{n})$. Likewise for $\rho(E)$ for a general downset $E$ of finite multisets of equations, replacing $\boldsymbol{\tau}$ by~$\boldsymbol{\rho}$. In other words, the right-hand side of the equivalence in Theorem~\ref{thm: algebraizability} is a precise counterpart of the ordinary definition of algebraizability via the translations $\boldsymbol{\tau}(x)$ and $\boldsymbol{\rho}(x, y)$, except $\boldsymbol{\tau}(x)$ and $\boldsymbol{\rho}(x, y)$ are now downsets of finite multisets of equations or formulas, rather than sets of equations or formulas.
\end{example}

\end{document}